\documentclass{amsart}
\usepackage{amssymb}
\usepackage{amsthm}
\usepackage{amsmath,amscd}
\usepackage[mathscr]{euscript}
\usepackage[all]{xy}
\usepackage{color}
\usepackage[utf8]{inputenc}
\usepackage{lmodern}
\usepackage[T1]{fontenc}
\usepackage[textwidth=15cm, hcentering]{geometry}
\usepackage[pagebackref=true,breaklinks=true,colorlinks]{hyperref}
\swapnumbers

\newtheorem{theo}{Theorem}[section]
\newtheorem{lemm}[theo]{Lemma}
\newtheorem{prop}[theo]{Proposition}
\newtheorem{coro}[theo]{Corollary}

\theoremstyle{definition}
\newtheorem{defi}[theo]{Definition}
\newtheorem{cons}[theo]{Construction}

\newtheorem{rem}[theo]{Remark}

\newtheorem*{theo*}{Theorem}

\numberwithin{equation}{section}

\newcommand{\op}{^{\mathrm{op}}}
\newcommand{\cat}{\mathbf}
\newcommand{\oper}{\mathscr}
\newcommand{\on}{\operatorname}
\newcommand{\Cat}{\cat{Cat}}
\newcommand{\Alg}{\cat{Alg}}
\newcommand{\id}{\mathrm{id}}
\newcommand{\R}{\mathbb{R}}
\renewcommand{\L}{\mathbb{L}}
\newcommand{\Map}{\on{Map}}
\newcommand{\map}{\on{map}}
\newcommand{\Fun}{\on{Fun}}

\newcommand{\lto}[1]{\stackrel{#1}{\longrightarrow}}

\newcommand{\h}{\widehat}
\newcommand{\Op}{\cat{Op}}
\newcommand{\WOp}{\cat{WOp}}
\newcommand{\G}{\cat{G}}
\newcommand{\pG}{\h{\cat{G}}}
\renewcommand{\S}{\cat{S}}
\newcommand{\pS}{\h{\cat{S}}}
\newcommand{\pGrp}{\h{\cat{Grp}}}
\newcommand{\Set}{\cat{Set}}
\newcommand{\pSet}{\h{\cat{Set}}}
\newcommand{\pab}{\oper{P}a\oper{B}}
\newcommand{\paub}{\oper{P}a\oper{UB}}
\newcommand{\pGT}{\h{\mathrm{GT}}}
\newcommand{\puGT}{\underline{\pGT}}
\newcommand{\sslash}{\mathbin{/\mkern-6mu/}}
\newcommand{\Pro}{\on{Pro}}
\renewcommand{\bf}{\mathbf}
\newcommand{\FF}{\mathbb{F}}

\title[Profinite completion of operads]{Profinite completion of operads and the Grothendieck-Teichm\"uller group}

\author{Geoffroy Horel}

\email{geoffroy.horel@gmail.com}

\address{Mathematisches Institut\\
Einsteinstrasse 62\\
D-48149 Münster\\
Deutschland}

\thanks{The author was supported by Michael Weiss's Humboldt professor grant.}



\begin{document}

\begin{abstract}
In this paper, we prove that the group of homotopy automorphisms of the profinite completion of the operad of little $2$-disks is isomorphic to the profinite Grothendieck-Teichm\"uller group. In particular, the absolute Galois group of $\mathbb{Q}$ acts faithfully on the profinite completion of $\oper{E}_2$ in the homotopy category of profinite weak operads.
\end{abstract}

\maketitle

\tableofcontents

\section*{Introduction}

The main result of this paper can be slightly imprecisely stated as follows:

\begin{theo*}[\ref{theo-main theorem spaces}]
The group of homotopy automorphisms of the profinite completion of the operad $\oper{E}_2$ of little $2$-disks is isomorphic to the Grothendieck-Teichm\"uller group.
\end{theo*}

We now introduce the main characters of this story.

\subsection*{Profinite completion}

Profinite completion of spaces has been introduced by Artin and Mazur in \cite{artinetale}. It is a homotopical analogue of the notion of profinite completion of groups. A space is said to be $\pi$-finite if it has finitely many path components and if, for any choice of base point, its homotopy groups based at that point are finite and almost all zero. For a general space $X$ the category of $\pi$-finite spaces with a map from $X$ fails to have an initial object in general. Nevertheless, there is an object in the pro-category of $\on{Ho}\S$ which plays the role of this missing universal $\pi$-finite space. This pro-object is the definition of the profinite completion of $X$ according to Artin and Mazur.

For our purposes this construction of Artin and Mazur is not sufficient because it gives a pro-object in the homotopy category of spaces and we need a point-set level lift of this object. More precisely, we need a category $\pS$ of profinite spaces ideally equipped with a model structure and a profinite completion functor $\h{(-)}:\S\to\pS$ ideally a left Quillen functor. We would also like a comparison map $\on{Ho}\pS\to\Pro(\on{Ho}\S)$ which maps $\h{X}$ to an object that is isomorphic to the Artin-Mazur profinite completion. A model structure fulfilling all these requirements has been constructed by Gereon Quick in \cite{quickprofinite}. There could however be several distinct profinite completion functors lifting Artin and Mazur's construction. The language of $\infty$-categories gives us a way to formulate precisely what profinite completion should be. In \cite{barneapro}, Barnea, Harpaz and the author prove that Quick's construction is ``correct'' in the sense that its underlying $\infty$-category is the $\infty$-category obtained by freely adjoining cofiltered limits to the $\infty$-category of $\pi$-finite spaces.

\subsection*{The little disk operad}
The little $2$-disk operad is an operad in topological spaces. It was introduced by May and Boardman-Vogt in order to describe the structure existing on the $2$-fold loops on a simply connected based space that allows one to recover that space up to weak equivalence (see \cite{maygeometry} for details about this theorem). The $n$-th space of the operad of little $2$-disks has the homotopy type of the space of configurations of $n$ points in $\mathbb{R}^2$. The latter space is well-known to be equivalent to the classifying space of the pure braid group on $n$ strands. This fact allows for the existence of groupoid models of $\oper{E}_2$. More precisely, there exist operads in groupoids which give a model of $\oper{E}_2$ when applying levelwise the classifying space functor. Two of these models called $\pab$ and $\paub$ play an important role in this work. The operad $\pab$ is the operad of parenthesized braids and is the operad controlling the structure of a braided monoidal category with a strict unit. It is the operad that enters in the definition of the Grothendieck-Teichm\"uller group as explained in the next paragraph. The operad $\paub$ of parenthesized unital braids is an explicit cofibrant replacement of $\pab$ (see \ref{coro-paub cofibrant}). It is also the operad controlling the structure of a braided monoidal category.

\subsection*{The Grothendieck-Teichm\"uller group}

The Grothendieck-Teichmüller group was introduced by Drinfel'd and Ihara following an idea of Grothendieck. Its story originates in Belyi's theorem. One consequence of this theorem is that the  action of $\on{Gal}(\bar{\mathbb{Q}}/\mathbb{Q})$ on the group of $\pi_1^{ét}(\bar{\mathbb{Q}}\times_{\mathbb{Q}}\mathcal{M}_{0,4})\cong \h{F_2}$ is faithful. Grothendieck's idea, explained in \cite{grothendieckesquisse}, was to use the rich structure that the collection of stacks $\mathcal{M}_{g,n}$ possess in order to understand the image of $\on{Gal}(\bar{\mathbb{Q}}/\mathbb{Q})$ in $\on{Aut}(\h{F_2})$. The ultimate hope was to find a finite presentation of $\on{Gal}(\bar{\mathbb{Q}}/\mathbb{Q})$. 

The étale fundamental group of $\mathcal{M}_{0,n}$ is a quotient of the braid group $B_n$ and these braid groups $B_n$ control the structure of braided monoidal category. Using these observations, Drinfel'd was able to construct in \cite{drinfeldquasi} a group $\pGT$ containing the group $\on{Gal}(\bar{\mathbb{Q}}/\mathbb{Q})$ and acting on the profinite completion of the braid groups in a way compatible with the Galois action on $\pi_1^{ét}(\bar{\mathbb{Q}}\times_{\mathbb{Q}}\mathcal{M}_{0,n+1})$. It is still unknown if the inclusion of $\on{Gal}(\bar{\mathbb{Q}}/\mathbb{Q})$ in $\pGT$ is an isomorphism.

Even if he does not use this language, Drinfel'd's construction of $\pGT$ is of operadic nature. As observed by Fresse in \cite{fressehomotopy1}, it relies in an essential way on the operad $\pab$ mentioned in the previous paragraph. Applying profinite completion on each level of the operad $\pab$ yields an operad in profinite groupoids $\h{\pab}$ whose $n$-th level is equivalent as a profinite groupoid to the profinite completion of the pure braid group on $n$ strands. One can rephrase Drinfel'd's definition by saying that the Grothendieck-Teichm\"uller group is the group of automorphisms of $\h{\pab}$ that induce the identity on objects. 

\subsection*{Content of the present paper}

Using Drinfel'd's definition of $\pGT$, our proof relies on the observation that $\h{\pab}$ is a groupoid model for the profinite completion of $\oper{E}_2$ and that the action of $\pGT$ on $\h{\pab}$ induces an isomorphism from $\pGT$ to the group of homotopy automorphisms of $\h{\pab}$. A technicality that we have to deal with is that the profinite completion functor from spaces to profinite spaces does not preserve products. Thus, applying profinite completion to each level of an operad does not yield an operad in general. To solve this problem we use the formalism of algebraic theories and their homotopy algebras initiated by Badzioch in \cite{badziochalgebraic}. This allows us to relax the axiom of operads and work with what we call weak operads. In the second section, we define the notion of weak operads in a reasonable model category and encode their homotopy theory by a model structure. We can then define the profinite completion functor as a functor from weak operads in spaces to weak operads in profinite spaces.

In this paper, we also study the automorphisms of the topological operad $\oper{E}_2$ before completion. There is a well-known action of the orthogonal group $\on{O}(2,\mathbb{R})$ on $\oper{E}_2$ and we prove in theorem \ref{theo-iso of E_2 topological} that the induced map from $\on{O}(2,\mathbb{R})$ to $\Map^h_{\Op\S}(\oper{E}_2,\oper{E}_2)$ is a weak equivalence. In particular, the group of connected components of $\Map^h_{\Op\S}(\oper{E}_2,\oper{E}_2)$ is isomorphic to $\mathbb{Z}/2$. 

\subsection*{Future work}
It has been conjectured that the operad $\oper{E}_2$ should have an algebro-geometric origin (see for instance the appendix of \cite{moravamotivic}). More precisely, there should exist an operad in a category of schemes (or a generalization thereof) over $\mathbb{Q}$ whose complex points form a model for $\oper{E}_2$. Applying the étale homotopy type to this conjectural operad would yield an operad in profinite spaces with an action of $\on{Gal}(\bar{\mathbb{Q}}/\mathbb{Q})$. Our main result seems to be a compelling evidence for this fact and we hope to tackle this problem in future work.

\subsection*{Related work}

We learned about this problem in a talk by Dwyer at the MSRI in 2014 (see \cite[56 min 40]{dwyervideo}). Our result should be compared to an analogous result due to Fresse (see \cite{fressehomotopy1} and \cite{fressehomotopy2}) which proves that the group of homotopy automorphisms of the rational completion of $\oper{E}_2$ is isomorphic to the pro-unipotent Grothendieck-Teichm\"uller group. The work of Sullivan on the Adams conjecture (see \cite{sullivangenetics}) especially the observation that $\on{Gal}(\bar{\mathbb{Q}}/\mathbb{Q})$ acts on the profinite completion of the spectrum $KU$ was a big influence on this work. The idea of using algebraic theories to relax the axiom of operads is an essential ingredient in this paper. This idea was initiated by Badzioch in \cite{badziochalgebraic} and continued by Bergner in \cite{bergnerrigidification} in the multi-sorted case. This work also relies a lot on good point set level models for profinite completion constructed by Morel \cite{morelensembles} and Quick \cite{quickprofinite}.

\section*{Acknowledgements}
I wish to thank Benoit Fresse for generously sharing some of his insights and for noticing a mistake in an earlier version of this paper. I am grateful to Benjamin Collas for a brilliant talk about Fresse's work in the Leray seminar in M\"unster that made me start thinking about this problem and for several helpful email exchanges about the Grothendieck-Teichm\"uller group. I also want to thank Ilan Barnea and Gereon Quick for helpful conversations and email exchanges.

\section*{Notations}
\begin{itemize}

\item For a category $\cat{C}$ we denote by $\cat{C}(x,y)$ the set of morphisms from $x$ to $y$ in $\cat{C}$. If the category is simplicially enriched, we denote by $\Map_{\cat{C}}(x,y)$ the mapping space from $x$ to $y$.

\item We generically denote by $\varnothing$ (resp. $\ast$) the initial (resp. terminal) object of a category $\cat{C}$. The category should be obvious from the context.

\item For $\cat{C}$ a model category, we denote by $\on{Ho}\cat{C}$ its homotopy category. The derived mapping spaces in $\cat{C}$ are denoted $\Map_{\cat{C}}^h(X,Y)$. They are only well-defined up to homotopy. They can be defined using Dwyer-Kan's hammock localization, or if $\cat{C}$ is a simplicial model category, by taking cofibrant-fibrant replacements of $X$ and $Y$. Note that $\on{Ho}\cat{C}(x,y)\cong\pi_0\Map^h_\cat{C}(x,y)$.

\item We denote an isomorphism by $\cong$ and a weak equivalence by $\simeq$.

\item We denote by $\S$ the category of simplicial sets, and $\G$ the category of groupoids. They are equipped respectively with the Kan-Quillen model structure and the canonical model structure.

\item For $k$ a non-negative integer, we denote by $I[k]$ the groupoid completion of the category $[k]$.

\item We denote by $\pS$ the category of simplicial objects in profinite sets equipped with Quick's model structure (see \cite{quickprofinite}) and by $\pG$ the category of profinite groupoids equipped with the model structure of theorem \ref{theo-model structure on profinite groupoids}.

\item For $\cat{C}$ a category with products, we denote by $\Op\cat{C}$ the category of operads in $\cat{C}$.

\item We denote by $\cat{POpC}$ the category of preoperads in $\cat{C}$ (i.e. the category of contravariant functors from $\Psi$ to $\cat{C}$ where $\Psi$ is the algebraic theory controlling operads). If $\cat{C}$ is a suitable model category, we denote by $\WOp\cat{C}$ the model category of weak operads in $\cat{C}$ (see proposition \ref{prop-model structure on weak operads}). The relevant definitions can be found in the second section.

\item We denote the little $2$-disks operad by $\oper{E}_2$. We implicitly see $\oper{E}_2$ as an operad in $\S$ rather than topological spaces. We denote by $E_2$ the weak operad in spaces $N^{\Psi}\oper{E}_2$. We denote by $\pab$ the operad of parenthesized braids (see construction \ref{cons-PaB}). 

\end{itemize}

\section*{Sketch of the proof}

It is actually more convenient to prove a slightly more general result. There is a monoid $\puGT$ defined by Drinfel'd whose group of units is $\pGT$ and we in fact prove that the endomorphisms of the profinite completion of $\oper{E}_2$ in the category of weak operads in profinite spaces is isomorphic to $\puGT$.

The four important categories in this work are $\S$, $\pS$, $\G$ and $\pG$. They are respectively the category of simplicial sets, profinite spaces (i.e. simplicial objects in profinite sets), groupoids and profinite groupoids (i.e. the pro-category of the category of groupoids with finitely many morphisms). Each of them has a model structure. The model structure on $\S$ and $\G$ are respectively the Kan-Quillen and canonical model structure. The model structure on $\pS$ is constructed by Quick in \cite{quickprofinite} and the model structure on $\pG$ is a groupoid analogue of Quick's model structure constructed in section \ref{section-model structure on profinite groupoids} of this paper. There is a classifying space functor $B$ from $\G$ to $\S$ and from $\pG$ to $\pS$. In both cases, $B$ is a right Quillen functor. There are also profinite completion functors $\h{(-)}:\G\to\pG$ and $\h{(-)}:\S\to\pS$ that are both left Quillen functors.

There is an operad $\pab$ in the category of groupoids which is a groupoid model of $\oper{E}_2$ in the sense that $B\pab$ is weakly equivalent to to $\oper{E}_2$. The levelwise profinite completion of $\pab$ is an operad $\h{\pab}$ in profinite groupoids. The monoid $\puGT$ is defined by Drinfel'd to be the monoid of endomorphisms of $\h{\pab}$ which induce the identity on objects.

There is a functorial path object in the category of profinite groupoids given by $C\mapsto C^{I[1]}$ where $I[1]$ denote the groupoid completion of the category $[1]$. This path object gives a notion of homotopies between maps of profinite groupoids. A levelwise application of this path object induces a path object in the category of operads in profinite groupoids. We denote by $\pi\Op\pG$ the category whose objects are operads in profinite groupoids and whose morphisms are homotopy classes of maps between them. The first main step in the proof is the following:

\begin{theo*}[\ref{theo-main theorem naive homotopy}]
The map $\puGT\to\on{End}(\h{\pab})$ induces an isomorphism
\[\puGT\to\on{End}_{\pi\Op\pG}(\h{\pab}).\]
\end{theo*}

One of the main issue with the profinite completion of spaces is that it does not preserve products. This led us to work with weak operads instead. A weak operad in a relative category with products $\cat{C}$ is a homotopy algebra in $\cat{C}$ over the algebraic theory $\Psi\op$ that controls operads.

In good cases, we construct a model category $\WOp\cat{C}$ encoding the homotopy category of weak operads in $\cat{C}$. The profinite completion of spaces or groupoids induces a left Quillen functor 
\[\h{(-)}:\WOp\S\to \WOp\pS\]
which we take as our definition of the profinite completion of an operad. There is a similar profinite completion left Quillen functor for weak operads in groupoids.

There is a operadic nerve functor $N^{\Psi}:\Op\pG\to\WOp\pG$. This operadic nerve is fully faithful and preserves the path object that exists on both sides. Thus we have an isomorphism
\[\on{End}_{\pi\Op\pG}(\h{\pab})\cong \on{End}_{\pi\WOp\pG}(N^{\Psi}\h{\pab}).\]

The endomorphisms of $N^{\Psi}\h{\pab}$ in $\on{Ho}\WOp\pG$ would coincide with the endomorphisms in $\pi\WOp\pG$ if $N^{\Psi}\h{\pab}$ was cofibrant and fibrant. The weak operad $N^{\Psi}\h{\pab}$ is not cofibrant, nevertheless, we prove the following:

\begin{theo*}[\ref{theo-main theorem for groupoids}]
The composite
\[\puGT\to\on{End}_{\pi\Op\pG}(\h{\pab})\to\on{End}_{\on{Ho}\WOp\pG}(N^{\Psi}\h{\pab})\]
is an isomorphism.
\end{theo*}

The last step is to lift this result about groupoid to a statement about spaces. This is not something that can be done in general because, for a groupoid $C$, the natural map $\h{BC}\to B\h{C}$ (where $B$ denotes the classifying space functor) is in general not an equivalence. More precisely, the completion of the classifying space of $C$ could have non-trivial homotopy groups in degree higher than $1$. Fortunately, this kind of pathology does not occur for the groupoids which appear in the operad $\pab$ and we can prove the following:

\begin{theo*}[\ref{theo-main theorem spaces}]
There is an isomorphism of monoids
\[\on{End}_{\on{Ho}\WOp\pG}(N^{\Psi}\h{\pab})\cong\on{End}_{\on{Ho}\WOp\pS}(\h{N^{\Psi}B\pab}).\]
\end{theo*}

In particular, since $\oper{E}_2\simeq B\pab$, we have an isomorphism
\[\puGT\cong\on{End}_{\on{Ho}\WOp\pS}(\h{E_2}).\]

There is an ambiguity on what the profinite completion of a space ought to be. For some authors, the profinite completion should be a pro-object in spaces. For other authors like Sullivan in \cite{sullivangenetics}, the profinite completion should be the inverse limit in spaces of that inverse system. More precisely, we have a right Quillen functor $|-|:\pS\to \S$ which takes a profinite space to its inverse limit in spaces. The profinite completion of a space $X$ could be defined as $\h{X}$ or as $|R\h{X}|$ where $R$ denotes a fibrant replacement in $\pS$. In general the right derived functor of $|-|$ is not  fully faithful. However, in the particular case that we are considering, we can prove the following variant of our main result:

\begin{theo*}[\ref{coro-alternative version of the main result}]
Let $R\h{E_2}$ be a fibrant replacement of $\h{E_2}$ in $\WOp\pS$. There is an isomorphism
\[\on{End}_{\on{Ho}\WOp\pS}(\h{E_2})\cong\on{End}_{\on{Ho}\WOp\S}(|R\h{E_2}|).\]
In particular, we also have an isomorphism
\[\puGT\cong\on{End}_{\on{Ho}\WOp\S}(|R\h{E_2}|).\]
\end{theo*}

\section{A few facts about model categories}

For future references, we recall a few useful facts about model categories.

\subsection*{Cofibrant generation}

\begin{defi}
Let $\cat{X}$ be a cocomplete category and $I$ a set of map in $\cat{X}$. 
\begin{itemize}
\item The \emph{$I$-cell complexes} are the smallest class of maps in $\cat{X}$ containing $I$ and closed under pushout and transfinite composition. 
\item The \emph{$I$-fibrations} are the maps with the right lifting property against $I$. 
\item The \emph{$I$-cofibrations} are the maps with the left lifting property against the $I$-fibrations. 
\end{itemize}
\end{defi}

It is easy to see that the $I$-fibrations are exactly the map with the right lifting property against the $I$-cofibrations. If the source of the maps of $I$ are small, then the $I$-cofibrations are exactly the retracts of the $I$-cell complexes. These facts can be found in appendix A of \cite{lurietopos}.

\subsection*{Mapping spaces and adjunctions}

As any category with weak equivalences, a model category has a simplicial enrichment given by the hammock localization. We denote by $\Map^h_{\cat{X}}(X,Y)$ the space of maps from $X$ to $Y$ in the hammock localization of $\cat{X}$ (see \cite[3.1.]{dwyerfunction} for a definition of the hammock localization). 

We denote by $\Map_{\cat{X}}(X,Y)$ the simplicial set of maps from $X$ to $Y$ whenever $\cat{X}$ has a natural enrichment in simplicial sets. This space is related to the previous space by the following theorem:

\begin{theo}
Let $\cat{X}$ be a simplicial model category, let $X$ be a cofibrant object and $Y$ be a fibrant object, then there is an isomorphism in $\on{Ho}(\S)$
\[\Map_{\cat{X}}(X,Y)\simeq \Map_{\cat{X}}^h(X,Y).\]
\end{theo}

\begin{proof}
See \cite[Corollary 4.7.]{dwyerfunction}.
\end{proof}

A Quillen adjunction is an adjunction up to homotopy in the following sense:

\begin{theo}\label{theo-Quillen adjunction and mapping spaces}
Let $F:\cat{X}\leftrightarrows\cat{Y}:U$ be a Quillen adjunction. Then we have an isomorphism in $\on{Ho}(\S)$
\[\Map^h_{\cat{Y}}(\L FX,Y)\simeq\Map^h_{\cat{X}}(X,\R UY).\]
\end{theo}

\subsection*{Left Bousfield localization}

\begin{defi}
Let $\cat{X}$ be a model category. A \emph{left Bousfield localization} of $\cat{X}$ is a model category $L\cat{X}$ whose underlying category is $\cat{X}$, whose cofibrations are the cofibrations of $\cat{X}$ and whose weak equivalences contain the weak equivalences of $\cat{X}$.
\end{defi}

Tautologically, if $L\cat{X}$ is a left Bousfield localization, the identity functor induces a Quillen adjunction
\[\id:\cat{X}\leftrightarrows L\cat{X}:\id.\]

\begin{defi}
Let $\cat{X}$ be a model category and $S$ be a class of maps in $\cat{X}$. Then we say that an object $Z$ of $\cat{X}$ is \emph{$S$-local} if for any map $u:A\to B$ in $S$, the induced map
\[\Map^h(B,Z)\to\Map^h(A,Z)\]
is a weak equivalence.

Dually, if $\mathcal{K}$ is a class of objects of $\cat{X}$, we say that a map $u:A\to B$ is a $\mathcal{K}$-weak equivalence if for all $Z$ in $\mathcal{K}$, the induced map
\[\Map^h(B,Z)\to\Map^h(A,Z)\]
is a weak equivalence
\end{defi}

\begin{rem}
Note that our definition of $S$-local objects differs slightly form that of \cite{hirschhornmodel}. An $S$-local object for Hirschhorn is an $S$-local object for us that is also fibrant.
\end{rem}

Let $S$ be a class of maps in $\cat{X}$. If it exists, we denote by $L_S\cat{X}$ the left Bousfield localization of $\cat{X}$ whose weak equivalences are the $\mathcal{K}$-equivalences for $\mathcal{K}$ the class of $S$-local objects.

It is usually hard to determine the fibrations of a Bousfield localization, however, the fibrant objects have a nice characterization:

\begin{prop}\label{prop-fibrant objects in a localization}
If $\cat{X}$ is left proper and $L_S\cat{X}$ exists, its fibrant objects are exactly the objects that are $S$-local and fibrant in $\cat{X}$.
\end{prop}

\begin{proof}
This is proved in \cite[Proposition 3.4.1.]{hirschhornmodel}. 
\end{proof}

\begin{prop}
Let $\cat{X}\to L\cat{X}$ be a left Bousfield localization of a left proper model category $\cat{X}$. Let $T$ be local with with respect to the weak equivalences of $L\cat{X}$ and $Z$ be any object. Then we have
\[\Map^h_{L\cat{X}}(Z,T)\simeq\Map^h_\cat{X}(Z,T).\]
\end{prop}

\begin{proof}
This is just theorem \ref{theo-Quillen adjunction and mapping spaces} applied to the Quillen adjunction
\[\id:\cat{X}\leftrightarrows L\cat{X}:\id.\]
\end{proof}

We have two theorems of existence of left Bousfield localizations. One in the combinatorial case and one in the cocombinatorial case. We recall that a combinatorial model category is a model category that is cofibrantly generated and whose underlying category is presentable. We say that a model category is cocombinatorial if the opposite model category is combinatorial.

\begin{theo}\label{theo-Left localization combinatorial}
Let $\cat{X}$ be a left proper combinatorial model category, let $S$ be a set of maps in $\cat{X}$, then there is a model structure on $\cat{X}$ denoted $L_S\cat{X}$ such that
\begin{itemize}
\item The cofibrations of $L_S\cat{X}$ are the cofibrations of $\cat{X}$.
\item The fibrant objects of $L_S\cat{X}$ are the objects of $\cat{X}$ that are both $S$-local and fibrant in $\cat{X}$
\item The weak equivalences in $L_S\cat{X}$ are the $\mathcal{K}$-equivalences for $\mathcal{K}$ the class of $S$-local objects of $\cat{X}$.
\end{itemize}
Moreover, this model structure is left proper and combinatorial. If $\cat{X}$ is tractable and simplicial, then $L_S\cat{X}$ is simplicial (for the same simplicial structure).
\end{theo}

Before stating the second theorem, let us recall from \cite{barwickleft} that a cocombinatorial model category $\cat{X}$ is said to be \emph{cotractable} if we can choose a set of generating fibration whose targets are fibrant.

\begin{theo}\label{theo-Left localization cocombinatorial}
Let $\cat{X}$ be a cotractable and left proper model category. Let $\mathcal{K}$ be a full subcategory of $\cat{X}$ such that
\begin{itemize}
\item The category $\mathcal{K}$ is a coaccessible and coaccessibly embedded subcategory of $\cat{X}$.
\item The category $\mathcal{K}$ is stable under homotopy limits in $\cat{X}$.
\item The category $\mathcal{K}$ is stable under weak equivalences.
\end{itemize}
Then there exists a model structure on $\cat{X}$ denoted $L_{\mathcal{K}}\cat{X}$ such that
\begin{itemize}
\item The cofibrations of $L_{\mathcal{K}}\cat{X}$ are the cofibrations of $\cat{X}$.
\item The fibrant objects of $L_{\mathcal{K}}\cat{X}$ are the objects of $\mathcal{K}$ that are fibrant in $\cat{X}$.
\item The weak equivalences are the $\mathcal{K}$-local equivalences.
\end{itemize}
\end{theo}

\begin{proof}
The dual statement is proved in \cite[Theorem 5.22]{barwickleft}.
\end{proof}

\section{Weak operads}

Let $\cat{C}$ be a category with finite products. We assume that the reader is familiar with the notion of operad. We denote by $\Op\cat{C}$ the category of operads in $\cat{C}$ with respect to the symmetric monoidal structure given by the cartesian product.

\subsection*{The algebraic theory of operads}

The theory of operads is definable by an algebraic theory as is observed in \cite[Example 3.4.]{bergnerrigidification}. In this subsection, we recall the details of this construction.

\begin{defi}
Let $S$ be a set. An \emph{$S$-sorted algebraic theory} is a category with products $\Phi$ whose objects are $T_{\bf{a}}$ for each finite sequence $\bf{a}=\{a_1,\ldots,a_n\}$ of elements of $S$. Moreover, we require the existence of an isomorphism
\[T_{\bf{a}}\cong T_{a_1}\times T_{a_2}\times\ldots\times T_{a_n}.\]
\end{defi}

\begin{defi}
Let $\Phi$ be an algebraic theory. Let $\cat{C}$ be a category with products. The category of $\Phi$-algebras in $\cat{C}$ is the category of product preserving functors from $\Phi$ to $\cat{C}$.
\end{defi}

Let $\mathbb{N}$ denote the set of nonnegative integers. There is a forgetful functor
\[U:\Op\Set\to\Set^{\mathbb{N}}\]
that sends an operad $\oper{O}$ to the collection $\{\oper{O}(n)\}_{n\geq 1}$. This functor has a left adjoint denoted $\FF$. Now, we construct an $\mathbb{N}$-sorted theory $\Psi\op$. First, we associate to a sequence $\bf{a}=\{a_1,\ldots,a_n\}$ the element $S_{\bf{a}}$ of $\Set^{\mathbb{N}}$ given by
\[S_{\bf{a}}(k)=\sqcup_{i, a_i=k}\ast.\]

Notice that we have an isomorphism
\begin{equation}\label{e-S preserves coprod}
S_{\bf{a}}\cong S_{a_1}\sqcup\ldots S_{a_n}.
\end{equation}

Now, we define $\Psi$ to be the category whose objects are $T_{\bf{a}}$ for $\bf{a}$ a finite sequence of integers and with morphisms
\[\Psi(T_{\bf{a}},T_{\bf{b}}):=\Op\Set(\FF S_{\bf{a}},\FF S_{\bf{b}}).\]
The composition in $\Psi$ is given by composition in $\Op\Set$. The category $\Psi$ is thus a full subcategory of the category $\Op\Set$. Equation \ref{e-S preserves coprod} implies that $\FF S_{\bf{a}}$ is isomorphic to the coproduct in $\Op\Set$ of the $\FF S_{a_i}$. This immediately implies that $\Psi\op$ is an $\mathbb{N}$-sorted algebraic theory.

\begin{prop}
The category of $\Psi\op$-algebras in sets is equivalent to the category $\Op\Set$.
\end{prop}

\begin{proof}
This proposition is well-known and we only sketch the proof. There is a functor $N^{\Psi}:\Op\Set\to\Alg^{\Psi\op}$ which sends $\oper{O}$ to $T_{\bf{a}}\mapsto \prod_{i}\oper{O}(a_i)$. This functor is clearly faithful. 

We introduce a simplifying notation. Given $\bf{a}$ a finite sequence of integers and $X$ an object  in $\Set^{\mathbb{N}}$, we denote by $X(\bf{a})$ the set $\Set^{\mathbb{N}}(S_{\bf{a}},X)$. 

Looking at the definition, we see that an operad is an object $\oper{O}$ of $\Set^{\mathbb{N}}$ equipped with a collection of operations of the form
\begin{equation}\label{e-operation}
\oper{O}(\bf{a})\to\oper{O}(\bf{b})
\end{equation}
that satisfy several relations which can all be expressed by saying that two maps
\begin{equation}\label{e-relation}
\oper{O}(\bf{c})\to\oper{O}(\bf{d})
\end{equation}
constructed from the operations are equal.

Since $\FF S_\bf{a}$ is the object representing $\oper{O}\mapsto \oper{O}(\bf{a})$ each of the operation \ref{e-operation} must be represented by a map $\FF S_\bf{b}\to\FF S_\bf{a}$. This implies that $N^\Psi$ is full. Indeed, a map of operads is a map of collections $\{\oper{O}(a)\}_{a\in\mathbb{N}}\to\{\oper{P}(a)\}_{a\in\mathbb{N}}$ commuting with the operations \ref{e-operation}. Since these operations are represented by maps in $\Psi$, any map of presheaves $N^\Psi\oper{O}\to N^{\Psi}\oper{P}$ restricts to a map $\oper{O}\to\oper{P}$.

Moreover, the relations satisfied by an operad are in particular valid for $\oper{O}=\FF S_\bf{e}$ for any finite sequence of integer $\bf{e}$. Thus by Yoneda's lemma in $\Psi$, the two maps $\FF S_\bf{d}\to\FF S_\bf{c}$ representing the relation \ref{e-relation} are equal. Hence, we see that given any functor $X:\Psi\op\to\Set$ preserving products, the collection $\{X(T_a)\}_{a\in\mathbb{N}}$ will satisfy the axioms of an operad. In other words, the functor $N^{\Psi}$ is essentially surjective. 
\end{proof}

\subsection*{Preoperads}

\begin{defi}
Let $\cat{C}$ be a category with finite products. We define the category $\cat{POpC}$ of \emph{preoperads in $\cat{C}$} to be the category of functors from $\Psi\op$ to $\cat{C}$. We define the category $\cat{OpC}$ to be the full subcategory of $\cat{POpC}$ spanned by the product preserving functors from $\Psi\op$ to $\cat{C}$.
\end{defi}

\begin{rem}
There is a slight conflict of notation with the previous subsection since the category $\Op\Set$ is not isomorphic to the category of product preserving functors $\Psi\op\to\cat{Set}$ but merely equivalent to it.
\end{rem}

We will denote the inclusion $\Op\cat{C}\to\cat{POpC}$ by the symbol $N^\Psi$ and call it the operadic nerve.

\begin{prop}\label{prop-preoperads main theorem}
Let $\cat{C}$ be a combinatorial (resp. cocombinatorial) model category. The category $\cat{POpC}$ has a combinatorial (resp. cocombinatorial) model structure in which a map is a weak equivalence (resp. a fibration) if it is objectwise a weak equivalence (resp. fibration). Moreover, this model structure is left proper if $\cat{C}$ is left proper and is simplicial if $\cat{C}$ is simplicial.
\end{prop}

\begin{proof}
This proposition has nothing to do with $\Psi$ and would be true for any functor category with a small source.

The existence of the model structure in the combinatorial case is very classical. If $\cat{C}$ is cocombinatorial, then $\cat{C}\op$ is combinatorial. Therefore, $\on{Fun}(\Psi,\cat{C}\op)$ admits the injective model structure by \cite[Proposition A.2.8.2.]{lurietopos} which dualizes to the projective model structure on $\on{Fun}(\Psi\op,\cat{C})=\cat{POpC}$.

The left properness follows from \cite[Remark A.2.8.4.]{lurietopos}

The simplicialness of $\cat{POpC}$ follows from \cite[Proposition A.3.3.2.]{lurietopos}. In the combinatorial case, we need to dualize but this is not a problem since the opposite of a simplicial model category is a simplicial model category.
\end{proof}

Assume that we have an adjunction
\[A:\cat{C}\leftrightarrows\cat{D}:B.\]
Then, applying $A$ and $B$ levelwise,  we get an adjunction
\[A:\cat{POp}\cat{C}\leftrightarrows\cat{POp}\cat{D}:B.\]

For future references we have the following very easy proposition.

\begin{prop}\label{prop-preoperads and adjunctions}
If $(A,B)$ is a Quillen adjunction, then the adjunction
\[A:\cat{POp}\cat{C}\leftrightarrows\cat{POp}\cat{D}:B\]
is a Quillen adjunction.
\end{prop}

\subsection*{The weak operads model structure}

Let $\cat{C}$ be a model category such that $\cat{POp}\cat{C}$ can be given the projective model structure. According to proposition \ref{prop-preoperads main theorem}, this happens for instance if $\cat{C}$ is combinatorial or cocombinatorial. 

For $X$ an object of $\cat{POpC}$ and $F:\Psi\op\to\Set$ a presheaf, we denote by $X(F)$ the object of $\cat{C}$ computed via the following end
\[X(F)=\int_{T_{\bf{a}}\in\Psi}X(T_{\bf{a}})^{F(T_{\bf{a}})}.\]
Alternatively, $F\mapsto X(F)$ is the unique colimit preserving functor $\Fun(\Psi\op,\Set)\to \cat{C}$ sending $\Psi(-,T_{\bf{a}})$ to $X(T_{\bf{a}})$. For $S$ a set and $K$ an element of $\cat{C}$, we denote by $S\boxtimes K$ the coproduct $\sqcup_SK$. For $F$ a presehaf on $\Psi$ and $K$ an element of $\cat{C}$, we denote by $K\boxtimes F$ the presheaf with value in $\cat{C}$ given by $T_{\bf{a}}\mapsto K\boxtimes X(T_{\bf{a}})$. We note that $K\boxtimes F$ is the object of $\on{Fun}(\Psi\op,\cat{C})$ representing the functor $X\mapsto \cat{C}(K,X(F))$.

Given $\bf{a}=\{a_1,\ldots,a_n\}$ an object in the category $\Psi$, there is an isomorphism 
\[\bigsqcup_{i}T_{a_i}\cong T_{\bf{a}}.\]
Thus, for any preoperad $X$ in $\cat{C}$, we get a map
\[s_{\bf{a},X}:X(\bf{a})\to\prod_i X(a_i)\]
that we call the Segal map.

\begin{defi}
We say that a fibrant object $X$ of $\cat{POp}\cat{C}$ is a \emph{weak operad}, if for any $T_{\bf{a}}$ in $\Psi$, the Segal maps
\[s_{\bf{a},X}:X(T_{\bf{a}})\to\prod_i X(a_i)\]
are weak equivalences. We say that a general object $X$ of $\cat{POpC}$ is a weak operad if one (and hence any) fibrant replacement of $X$ is a weak operad.
\end{defi}

Now we assume that $\cat{C}$ is left proper and either combinatorial or cocombinatorial.

\begin{prop}\label{prop-model structure on weak operads}
There is a model structure on $\cat{POpC}$ in which 
\begin{itemize}
\item The cofibrations are the cofibrations of $\cat{POpC}$.
\item The fibrant objects are the weak operads that are fibrant in $\cat{POpC}$.
\item The weak equivalences are the maps $f:X\to Y$ such that for any weak operad $Z$, the induced map
\[\Map^h_{\cat{POpC}}(Y,Z)\to\Map^h_{\cat{POpC}}(X,Z)\]
is a weak equivalence.
\end{itemize}
\end{prop}

\begin{proof}
In the combinatorial case, we can use theorem \ref{theo-Left localization combinatorial}. We need to specify a set of maps $S$ such that the $S$-local object are the weak operads. Let $\kappa$ be a regular cardinal such that $\cat{C}$ is $\kappa$-presentable and such that the $\kappa$-filtered colimits are homotopy colimits (it exists by \cite[Proposition 2.5.]{barwickleft}). Let $\mathcal{G}$ be a set of objects containing at least one representative of each isomorphism class of $\kappa$-compact object of $\cat{C}$. Let $Q$ be a cofibrant replacement functor in $\cat{C}$.

Let us consider the set $S$ of maps
\[\sqcup_{i}\Psi(-,T_{a_i})\boxtimes QK\to \Psi(-,T_{\bf{a}})\boxtimes QK\]
for any $\bf{a}$ and any $K$ in $\mathcal{G}$.

We claim that a fibrant object $X$ of $\cat{POpC}$ is a weak operad if and only if it is local with respect to $S$. Indeed, an object $X$ is $S$-local if and only if for each $\bf{a}$ and any $K$, the map
\[\Map_{\cat{C}}^h(K,X(T_{\bf{a}}))\to\prod_{i}\Map_{\cat{C}}^h(K,X(T_{a_i}))\]
is a weak equivalence. Thus, if $X$ is a weak operad, it is $S$-local. Conversely, Let $L$ be an object of $\cat{C}$. Since $\kappa$-filtered colimits are homotopy colimits, $L$ is weakly equivalent to $\on{hocolim}_IK_i$ where the $K_i$ are in $\mathcal{G}$ and $I$ is $\kappa$-filtered. Therefore, using the fact that $X$ is fibrant, we find that
\[\Map_{\cat{C}}^h(L,X(T_{\bf{a}}))\to\prod_{i}\Map_{\cat{C}}^h(L,X(T_{a_i}))
\simeq\Map_{\cat{C}}^h(L,\prod_{i}X(T_{a_i}))\]
is a weak equivalence. Since this is true for each $L$, this implies that the Segal maps for $X$ are weak equivalences.

\medskip
In the cocombinatorial case, we use theorem \ref{theo-Left localization cocombinatorial} to prove the existence of the model structure. We take as $\mathcal{K}$, the full subcategory of weak operads in $\cat{C}$. It is clear that $\mathcal{K}$ is stable under weak equivalences and homotopy limits. Thus it suffices to prove that $\mathcal{K}$ is coaccessible and coaccessibly embedded.

By proposition \ref{prop-preoperads main theorem} $\cat{POpC}$ is cocombinatorial. Therefore, by \cite[Proposition 2.5.]{barwickleft}, there exists a cardinal $\kappa$ such that $\kappa$-cofiltered limits are homotopy limits, the limit of a $\kappa$-cofiltered diagram in $\mathcal{K}$ is in $\mathcal{K}$ and thus is the limit in $\mathcal{K}$. Therefore $\mathcal{K}$ has $\kappa$-cofiltered limits and the inclusion $\mathcal{K}\to\cat{POpC}$ preserves those limits.

Now we check that $\mathcal{K}$ is coaccessible. Let us pick a $\kappa$-coaccessible fibrant replacement functor $R$. This can be done by \cite[Proposition 2.5.]{barwickleft} We have a map
\[P:\cat{POpC}\to \prod_{T_{\bf{a}}\in\Psi}\cat{C}^{[1]}\]
whose component indexed by $T_{\bf{a}}$ sends $X$ to the map
\[s_{\bf{a},RX}:RX(T_\bf{a})\to\prod_i RX(R_{a_i}).\] 

By definition, a weak operad is an object which is mapped by $P$ in the category $\prod_{T_{\bf{a}}\in\Psi}w\cat{C}^{[1]}\subset \prod_{T_{\bf{a}}\in\Psi}\cat{C}^{[1]}$.

Taking maybe a bigger $\kappa$, the category $\prod_{T_{\bf{a}}\in\Psi}w\cat{C}^{[1]}$ is $\kappa$-coaccessible and $\kappa$-coaccessibly embedded in $\prod_{T_{\bf{a}}\in\Psi}\cat{C}^{[1]}$, thus by \cite[Corollary A.2.6.5]{lurietopos}, $\mathcal{K}$ forms a $\kappa$-coaccessible category.
\end{proof}

\subsection*{Operads vs weak operads in spaces}

The category of operads in spaces admits a model structure in which the weak equivalences and fibrations are levelwise. This follows from \cite[Theorem 3.2.]{bergeraxiomatic} or the more general \cite[Theorem 4.7.]{bergnerrigidification}. The operadic nerve functor from operads in spaces to preoperads has a left adjoint $S$. Since $\Op\S$ is a simplicial category with all colimits and $\cat{POpS}$ is a presheaf category, there exists a unique colimit preserving simplicial functor $S$ from $\cat{POpS}$ to $\Op\S$ that sends the object represented by $T_{\bf{a}}$ to $\FF\bf{a}$. It is then obvious that this $S$ is a left adjoint for $N^{\Psi}$. 

The adjunction
\[S:\cat{POpS}\leftrightarrows\Op\S:N^{\Psi}\]
is a simplicial Quillen adjunction. Indeed, the functor $N^{\Psi}$ obviously preserves fibrations and trivial fibrations. If $f:X\to Y$ is a weak equivalence in $\WOp\S$, then, for any fibrant object $\oper{O}$ in $\Op\S$, the map
\[\Map_{\Op\S}^h(\L SY,\oper{O})\to\Map^h_{\Op\S}(\L SX,\oper{O})\]
induced by $f$ coincides by theorem \ref{theo-Quillen adjunction and mapping spaces} with the map
\[\Map^h_{\cat{POpS}}(Y,N^{\Psi}\oper{O})\to\Map^h_{\cat{POpS}}(X,N^{\Psi}\oper{O}).\]
Thus since $N^{\Psi}\oper{O}$ is a weak operad, we see that if $f$ is a weak equivalence in $\WOp\S$, then $\L S(f)$ is one in $\Op\S$. This implies that we have an induced Quillen adjunction
\[S:\WOp\S\leftrightarrows \Op\S:N^{\Psi}.\]

\begin{theo}\label{theo-weak operads vs operads in spaces}
The Quillen pair $(S,N^{\Psi})$ is in fact a Quillen equivalence
\end{theo}

\begin{proof}
This is \cite[Theorem 5.13.]{bergnerrigidification}
\end{proof}

\subsection*{Operads vs weak operads in groupoids}

Our goal in this subsection is to prove that the homotopy theory of weak operads in groupoids is equivalent to the homotopy theory of strict operads.

We have an adjunction
\[S:\cat{POp}\G\leftrightarrows \Op\G:N^{\Psi}.\]
As in the case of spaces, we construct $S$ as the unique $\G$-enriched functor which sends the presheaf represented by $T_{\bf{a}}$ to $\FF\bf{a}$ seen as an operad in $\G$ via the product preserving functor $\on{Disc}:\Set\to\G$. Exactly as in the case of spaces, it induces a Quillen adjunction
\[S:\WOp\G\leftrightarrows \Op\G:N^{\Psi}.\]

\begin{prop}\label{prop-weak operads vs operads in groupoids}
This Quillen adjunction is a Quillen equivalence.
\end{prop}

\begin{proof}
Let $\S_{\leq 1}$ be the localization of $\S$ at the map $\partial \Delta[2]\to\Delta[2]$. The model category $\S_{\leq 1}$ being combinatorial, we can form the model category $\WOp\S_{\leq 1}$. The cofibrations in this model structure are the cofibrations of $\WOp\S$ and the fibrant objects are the weak operads in $\S$ that are levelwise fibrant in $\S_{\leq 1}$. 

We can also form the model category $\Op\S_{\leq 1}$. This is a model category in which the cofibrations are the cofibrations in $\Op\S$ and the fibrant objects are the operads that are levelwise in $\S_{\leq 1}$. The existence of this model structure follows from \cite[Theorem 3.1.]{bergeraxiomatic}. A symmetric monoidal fibrant replacement functor in $\S_{\leq 1}$ is given by $X\mapsto B\pi(X)$.

We have a Quillen adjunction
\[S:\WOp\S_{\leq 1}\leftrightarrows \Op\S_{\leq 1}:N^{\Psi}.\]
We claim that this is a Quillen equivalence. If $\oper{O}$ is a fibrant object of $\Op\S_{\leq 1}$, then the counit map $SQN^{\Psi}\oper{O}\to \oper{O}$ is a weak equivalence in $\Op\S$. This follows from the fact that the cofibrant replacement in $\Op\S_{\leq 1}$ is a cofibrant replacement in $\Op\S$ and the fact that the adjunction $\WOp\S\leftrightarrows\Op\S$ is a Quillen equivalence.

Similarly, let $X$ be cofibrant in $\WOp\S_{\leq 1}$. Let $RSX$ be a fibrant replacement of $SX$ in $\Op\S$. Let $R_1SX$ be a fibrant replacement of $RSX$ in $\Op\S_{\leq 1}$. We have a map $RSX\to R_1SX$ which is in each degree a weak equivalence in $\S_{\leq 1}$. Applying $N^{\Psi}$, we get a weak equivalence in $\WOp\S_{\leq 1}$
\[N^{\Psi}RSX\to N^{\Psi}R_1SX.\]
Since we already know by theorem \ref{theo-weak operads vs operads in spaces} that $X\to N^{\Psi}RSX$ is a weak equivalence in $\WOp\S$, we have shown that the derived unit
\[X\to N^{\Psi}R_1SX\]
is a weak equivalence in $\WOp\S_{\leq 1}$.

We also have a Quillen equivalence $\pi:\S_{\leq 1}\leftrightarrows \G:B$ whose right adjoint is the classifying space functor. This induces a commutative square of left Quillen functors
\[
\xymatrix{
\WOp\S_{\leq 1}\ar[r]^S\ar[d]_{\pi}& \Op\S_{\leq 1}\ar[d]^{\pi}\\
\WOp\G\ar[r]_S& \Op\G
}
\]
We know that all maps except maybe the bottom horizontal map are Quillen equivalences. This forces the bottom horizontal map to be a Quillen equivalence.
\end{proof}

We can also prove that the functor $B:\WOp\G\to\WOp\S$ is homotopically fully faithful.

\begin{prop}\label{prop-fully faithfulness of B}
Let $X$ and $Y$ be two fibrant object of $\WOp\G$, then the map
\[\Map_{\WOp\G}^h(X,Y)\to\Map_{\WOp\S}^h(BX,BY)\]
is a weak equivalence
\end{prop}

\begin{proof}
According to the previous proposition, we have a sequence of Quillen adjunctions
\[\WOp\S\leftrightarrows\WOp\S_{\leq 1}\leftrightarrows\WOp\G\]
where the first is a localization and the second is an equivalence.

Thus we are reduced to proving that
\[\Map^h_{\WOp\S_{\leq 1}}(BX,BY)\to\Map^h_{\WOp\S}(BX,BY)\]
is a weak equivalence which is true by definition of a Bousfield localization since $BX$ and $BY$ are fibrant in $\S_{\leq 1}$.
\end{proof}

\section{Pro categories}

In this section, we recall a few basic facts about pro-categories.

\begin{defi}
A category $I$ is \emph{cofiltered} if for any finite category $K$ with a map $f:K\to I$, there exists an extension of $f$ to a cocone $K^{\triangleleft}\to I$.
\end{defi}

For any small category $C$, one can form the category $\on{Pro}(C)$ by formally adding cofiltered limits to $C$. More explicitly, the objects of $\Pro(C)$ are pairs $(I,X)$ where $I$ is a cofiltered small category and $X:I\to C$ is a functor. We usually write $\{X_i\}_{i\in I}$ for an object of $\Pro(C)$. The morphisms are given by
\[\Pro(C)(\{X_i\}_{i\in I},\{Y_j\}_{j\in J})=\on{lim}_{j\in J}\on{colim}_{i\in I} C(X_i, Y_j).\]

The category $\Pro(C)$ can be alternatively defined as the opposite of the full subcategory of $\on{Fun}(C,\Set)$ spanned by objects that are filtered colimits of representable functors. The equivalence with the previous definition comes from identifying $\{X_i\}_{i\in I}$ with the colimit of the diagram
\[i\mapsto C(X_i,-)\]
seen as an object of $\on{Fun}(C,\Set)\op$.

Note that there is an obvious fully faithful inclusion $C\to\Pro(C)$ sending $X\in C$ to the functor $C\to\cat{Set}$ represented by $X$. Moreover, it can be showed that $\Pro(C)$ has all cofiltered limits. In particular, if $i\mapsto X_i$ is a cofiltered diagram in $C$, its inverse limit in $\Pro(C)$ is $\{X_i\}_{i\in I}$.

The universal property of the pro-category can then be expressed in the following theorem.

\begin{theo}
Let $C$ be a small category and $D$ be a category with cofiltered limits. We denote by $\on{Fun}'(\Pro(C),D)$ the category of functors $\Pro(C)\to D$ which commute with cofiltered limits. Then the restriction functor
\[\on{Fun}'(\Pro(C),D)\to\on{Fun}(C,D)\]
is an equivalence of categories.
\end{theo}

\begin{prop}\label{prop-co presentability of pro categories}
If $C$ is a small finitely complete category, then $\Pro(C)$ is a copresentable category (i.e. the opposite category is a presentable category).
\end{prop}

\begin{proof}
The category $\Pro(C)\op$ is equivalent to $\on{Ind}(C\op)$ and $C\op$ has all finite colimits. Therefore $\on{Ind}(C\op)$ is presentable.
\end{proof}

If $C$ has all finite limits, we have a nice characterization of the filtered colimits of representable functors.

\begin{prop}
If $C$ has all finite limits, then a functor $C\to\cat{Set}$ is a filtered colimit of representable functors if and only if it preserves finite limits.
\end{prop}

\begin{proof}
Clearly all representable functors $C\to\cat{Set}$ preserve finite limits. In the category of sets finite limits commute with filtered colimits. This implies that any filtered colimit of representable functors preserves finite limits.

Conversely, as any covariant functor, $F$ is the colimit of the composite
\[C\op_{/F}\to C\op\to\on{Fun}(C,\Set)\]
where the second map is the Yoneda embedding. Thus it suffices to prove that $C\op_{/F}$ is filtered if $F$ preserves finite limits.

Let $I$ be a finite category and $u:I\to C\op_{/F}$ be a diagram in $C\op_{/F}$. In other words, $u$ is the data of a functor $v:I\to C\op$ with a map to the constant functor $I\to\on{Fun}(C,\Set)$ with value $F$. Since $F$ commutes with finite limits, the colimit of $v$ in $C\op$ (which is the limit of $v\op:I\op\to C$) has a natural map to $F$ which makes it a cocone for $u:I\to C\op_{/F}$.
\end{proof}

\begin{rem}\label{rem-pro representable functors}
In other words,  a functor $F:C\to\cat{Set}$ preserves finite limits if and only if has an extension $\Pro(C)\to\cat{Set}$ which is a representable functor. Moreover, any two choices of representing objects are canonically isomorphic. The situation can be summarized by saying that covariant functors that preserve finite limits are pro-representable. 
\end{rem}

\subsection*{Profinite sets and groups}

\begin{defi}
Let $\cat{F}$ be the category of finite sets. The category $\pSet$ of \emph{profinite sets} is defined to be the category $\Pro(\cat{F})$.
\end{defi}

Since the category $\cat{F}$ has all finite limits, the category $\pSet$ is the opposite of the category of finite limit preserving functors $\cat{F}\to\Set$. There is a more concrete way of understanding the category $\pSet$.

\begin{prop}
The category $\pSet$ is equivalent to the category of compact Hausdorff totally disconnected spaces and continuous maps.
\end{prop}

The functor from $\pSet$ to topological spaces is obtained by first considering a cofiltered diagram in finite sets as a cofiltered diagram in discrete topological spaces and then take its inverse limit in the category topological space.

Similarly, we can consider the category $f\cat{Grp}$ of finite groups and form the category
\[\h{\cat{Grp}}:=\Pro(f\cat{Grp}).\]

\begin{prop}
The category $\h{\cat{Grp}}$ is equivalent to the category of group objects in $\pSet$. Equivalently, the category $\h{\cat{Grp}}$ is the category of topological groups whose underlying topological space is compact Hausdorff and totally disconnected.
\end{prop}

There is a functor
\[\h{\cat{Grp}}\to\cat{Grp}\]
which sends a profinite group to its underlying group (forgetting the topology). This functor has a left adjoint $G\mapsto \h{G}$ called profinite completion.

\begin{defi}
The \emph{profinite completion of a discrete group $G$} denoted $\h{G}$ is the inverse limit of the diagram of topological groups
\[N\mapsto G/N\]
where $N$ runs over the poset of normal finite index subgroups of $G$ and $G/N$ is given the discrete topology.
\end{defi}

\section{Profinite groupoids}\label{section-model structure on profinite groupoids}

We first introduce a few useful notations. A groupoid, is a small category in which all morphisms are invertible. We denote by $\G$ the category of groupoids and by $\on{Ob}:\G\to\Set$ the functor sending a groupoid to its set of objects. This functor has a left adjoint $\on{Disc}$ which sends a set $S$ to the discrete groupoid on that set of objects (a groupoid is discrete if it has only identities as morphisms) and a right adjoint $\on{Codisc}$ which sends a set $S$ to the groupoid $\on{Codisc}(S)$ whose set of objects is $S$ and with exactly one morphisms between any two objects. We do not usually write the functor $\on{Disc}$  and see a set as a groupoid via this functor.

Given a set $S$ with a right action by a group $G$, we denote by $S\sslash G$ the translation groupoid. This is the groupoid whose set of objects is $S$. The set of morphisms from $s$ to $t$ is the set of elements of $G$ such that $s.g=t$. In particular, $\ast\sslash G$ is a groupoid whose nerve is the classifying space of $G$.

Given a set $S$ and a group $G$, we denote by $G[S]$ the groupoid $G\times\on{Codisc}(S)$. Note that any connected groupoid is non-canonically isomorphic to $G[S]$ with $S$ the set of objects and $G$ the group of automorphisms of a chosen object. A general groupoid $C$ is isomorphic to a disjoint union $\sqcup_{u\in \pi_0(C)} G_u[S_u]$ indexed by the set of connected component of $C$.

\subsection*{Profinite groupoids}
In this section, we construct a model structure on the category of profinite groupoids (i.e. the pro-category of the category of finite groupoids) that is analogous to Quick's model structure on $\pS$.

We say that a groupoid is a \emph{finite groupoid} if its set of morphisms is finite. Note that this also implies that the set of objects is finite. We denote by $f\G$ the full subcategory of $\cat{G}$ spanned by the finite groupoids. We denote by $\pG$ the pro-category of $f\G$.

\begin{defi}
Let $A$ be a finite groupoid and $S$ be a finite set. The \emph{$0$-th cohomology set} of $A$ with coefficients in $S$ is the set of maps $u:\on{Ob}(A)\to S$ that are constant on isomorphisms classes.
\end{defi}

Let $G$ be a finite group. We define the set $Z^1(A,G)$ to be the set of maps $u:\on{Ar}(A)\to G$ such that
\[u(f\circ g)=u(f)u(g).\]

We define $B^1(A,G)$ to be the set of maps $\phi:\on{Ob}(A)\to G$. The set $B^1(A,G)$ is a product of copies of $G$ and as such it has a group structure. There is a right action of $B^1(A,G)$ on $Z^1(A,G)$. Given $u$ in $Z^1(A,G)$ and $\phi$ in $B^1(A,G)$, we define $u.\phi$ in $Z^1(A,G)$ by the following formula:
\[(u.\phi)(f)=\phi(t(f))^{-1}u(f)\phi(s(f))\]
where $s$ and $t$ send a morphism in $A$ to its source and target.

\begin{defi}
The \emph{first cohomology} set of $A$ with coefficients in $G$ denoted $H^1(A,G)$ is the quotient $Z^1(A,G)/B^1(A,G)$.
\end{defi}

Now we give an alternative definition of $H^1(A,G)$. 

We write $I[1]$ for the codiscrete groupoid on two objects. Equivalently, $I[1]$ is the groupoid representing the functor $\G\to\Set$ sending $A$ to $\on{Ar}(A)$. For $G$ a finite group, we can form the $G$-set $G^c$ which is $G$ with the right action given by conjugation
\[g.h:=h^{-1}gh.\]

We observe that $Z^1(A,G)=f\cat{G}(A,\ast\sslash G)$.

There is a map $G^c\sslash G\to (\ast\sslash G)^2$. On objects, it is the unique map and it sends a conjugation $k^{-1}gk=h$ to $g$ and $h$ respectively. This map represents a pair of parallel maps
\[f\cat{G}(A,G^c\sslash G)\rightrightarrows f\cat{G}(A,\ast\sslash G)\]
for any finite groupoid $A$.

\begin{prop}
The coequalizer of 
\[f\cat{G}(A,G^c\sslash G)\rightrightarrows Z^1(A,G)\]
is isomorphic to $H^1(A,G)$.
\end{prop}

\begin{proof}
This is a trivial computation.
\end{proof}

\begin{defi}
Let $A=\{A_i\}_{i\in I}$ be a profinite groupoid. We define the \emph{$0$-th cohomology set} of $A$ with coefficients in a finite set $S$ by the formula: 
\[H^0(A,S)=\on{colim}_I H^0(A_i,S)\]
similarly, we define the \emph{first cohomology set} of $A$ with coefficients in a finite group $G$ by the formula
\[H^1(A,G)=\on{colim}_I H^1(A_i,G).\]
\end{defi}

\begin{lemm}\label{lemm-filtered colimits and quotients}
If $i\mapsto S_i$ is a filtered colimit of sets and $i\mapsto G_i$ is a filtered colimit of groups and if there is an action $S_i\times G_i\to S_i$ which is functorial in $i$, then, the obvious map
\[\on{colim}_i(S_i/G_i)\to (\on{colim}_i S_i)/(\on{colim}_i G_i)\]
is an isomorphism. 
\end{lemm}

\begin{proof}
For each $i$, we have a coequalizer diagram
\[S_i\times G_i\rightrightarrows S_i\to S_i/G_i.\]
Since filtered colimits commute with coequalizers and filtered colimits in groups are reflected by the forgetful functor to $\Set$, we are done.
\end{proof}

For $A=\{A_i\}_{i\in I}$ a profinite groupoid, we can define the set $Z^1(A,G)$ by the formula
\[Z^1(A,G)=\on{colim}_IZ^1(A_i,G)\]
and we can define $B^1(A,G)$ by a similar colimit. According to the previous lemma, we have
\[H^1(A,G):=\on{colim}_IH^1(A_i,G)\cong (\on{colim}_IZ^1(A_i,G))/(\on{colim}_IB^1(A_i,G))\cong Z^1(A,G)/B^1(A,G).\]

\begin{prop}\label{prop-representability of cohomology}
Let $S$ be a finite set and $G$ be a finite group. For any profinite groupoid $A$, there are isomorphisms:
\begin{align*}
H^0(A,S)&=\pG(A,S)\\
Z^1(A,G)&=\pG(A,\ast\sslash G)\\
B^1(A,G)&=\pG(A,G)
\end{align*}
\end{prop}

\begin{proof}
Each formula is true if $A$ is a finite groupoid. Moreover, by definition of the hom sets in a pro-category, we have $\pG(A,S)=\on{colim}_If\G(A_i,S)=H^0(A,S)$ and similarly in the other two cases.
\end{proof}

\begin{prop}\label{prop-cohomology of cofiltered limit}
Let $A:I\to \pG$ be a cofiltered diagram with value in profinite groupoids. Let $S$ be a finite set, then the map
\[\on{colim}_I H^0(A_i,S)\to H^0(\on{lim}_I A_i,S)\]
is an isomorphism. The obvious analogous statement holds for $H^1$.
\end{prop}

\begin{proof}
The case of $H^0$ is easy since $H^0(A,S)=\pG(A,S)$ and $S$ being an object of $f\G$ is cosmall in $\pG=\Pro(f\G)$. Similarly, $Z^1(-,G)$ and $B^1(-,G)$ are representable by objects of $f\G$ and thus they send cofiltered limits to filtered colimits. The result then follows from lemma \ref{lemm-filtered colimits and quotients}.
\end{proof}

\begin{defi}\label{defi-weak equivalences of profinite groupoids}
We say that a map $u:A\to B$ is $\pG$ is a \emph{weak equivalence} if for all finite set $S$, 
\[u^*:H^0(B,S)\to H^0(A,S)\]
is an isomorphism and for all finite group $G$
\[u^*:H^1(B,G)\to H^1(A,G)\]
is an isomorphism.
\end{defi}

\begin{prop}\label{prop-level weak equivalences}
Let $I$ be a cofiltered category and $A:I\to f\G$ and $B\to f\G$ be two functors and $u:A\to B$ be a natural transformation such that for all $i$, the map $u_i:A_i\to B_i$ is an equivalence of groupoid. Then the maps $A\to B$ is an equivalence in $\pG$.
\end{prop}

\begin{proof}
Let $S$ be a finite set. For all $i$, the map
\[H^0(B_i,S)\to H^0(A_i,S)\]
is an isomorphism. Therefore, the map:
\[H^0(B,S)\to H^0(A,S)\]
is an isomorphism as a colimit of isomorphism.

A similar proof holds for the first cohomology sets.
\end{proof}

\begin{prop}\label{prop-weak equivalences stable under cofiltered limits}
Weak equivalences in $\pG$ are stable under cofiltered limits.
\end{prop}

\begin{proof}
The proof is similar to the proof of the previous proposition and uses proposition \ref{prop-cohomology of cofiltered limit}.
\end{proof}

There is a functor $\on{Disc}:\cat{F}\to f\G$ sending the set $S$ to the discrete groupoid on that set of object. This functor has a left adjoint $\pi_0:f\G\to\cat{F}$. We can extend both functors to the pro-category by imposing that they commute with cofiltered limits and we get an adjunction:
\[\pi_0:\pG\leftrightarrows\pSet:\on{Disc}.\]

\begin{prop}\label{prop-equivalence induce iso on pi zero}
Let $f:A\to B$ be a weak equivalence in $\pG$. Then $\pi_0(f)$ is an isomorphism. 
\end{prop}

\begin{proof}
Let $S$ be a finite set. Then we have
\[H^0(A,S)\cong\pG(A,S)\cong\pSet(\pi_0(A),S).\]

Thus, the map $\pi_0(A)\to \pi_0(B)$ induces an isomorphism when mapping to a finite set. This is enough to insure that this is an isomorphism of profinite sets.
\end{proof}

\subsection*{Construction of the model structure}

We define two sets of arrows $P$ and $Q$ in $f\G$.

Let us pick a set $\mathcal{S}$ of finite sets containing a representative of each isomorphism class of finite set. Let $\mathcal{G}$ be the set of groups whose underlying set is in $\mathcal{S}$.

The set $P$ is the set of maps of the form:
\[G\sslash G\to \ast\sslash G, G^c\sslash G\to (\ast\sslash G)^2,\; \ast\sslash G\to \ast,\;S\to \ast,\; S\to S\times S\]
where $G$ is any finite group in $\mathcal{G}$ and $S$ is any finite set in $\mathcal{S}$.

The set $Q$ is the set of maps:
\[G\sslash G\to \ast\]
where $G$ is any finite group in $\mathcal{G}$.

We can now state the main theorem of this section:

\begin{theo}\label{theo-model structure on profinite groupoids}
The category $\pG$ has a cocombinatorial model structure in which the cofibrations (resp. trivial cofibrations) are the $Q$-projective maps (resp. $P$-projective maps) and the weak equivalences are as in definition \ref{defi-weak equivalences of profinite groupoids}.
\end{theo}

\begin{proof}
We apply the dual of \cite[Theorem 11.3.1]{hirschhornmodel}.

(1) The objects $\ast\sslash G$ and $S\times S$ are cosmall.

(2) The $Q$-cocell complexes are weak equivalences. Since weak equivalences are stable under cofiltered limits by proposition \ref{prop-weak equivalences stable under cofiltered limits}, it suffices to check that any pullback of a map in $Q$ is a weak equivalence. Let $A=\{A_i\}_{i\in I}$ be an object of $\pG$, then the map
\[A\times G\sslash G\to A\]
is the limit in $\pG$ of the maps
\[A_i\times G\sslash G\to A_i.\]
Therefore, it is a weak equivalences in $\pG$ by proposition \ref{prop-level weak equivalences}.

(3) The maps in $Q$ are $P$-cocell complexes.

(4) The $P$-projective maps are $Q$-projective. Indeed, the $P$-projective maps are those with the left lifting property against the $P$-cocell complexes. In particular, according to the previous paragraph, they have the left lifting property against the maps in $Q$ and hence are $Q$-projective.

(5) The $P$-projective maps are weak equivalences. Let $u:A\to B$ be a $P$-projective map. The left lifting property against the map $S\to \ast$ tells us that 
\[u^*:H^0(B,S)\to H^0(A,S)\]
is surjective. 

We clearly have the isomorphism of functors $H^0(-,S\times S)\cong H^0(-,S)^2$. The left lifting property against $S\to S\times S$ tells us that
\[H^0(B,S)\to H^0(A,S)\times_{H^0(A,S)^2} H^0(B,S)^2\]
is surjective which is saying that two classes in $H^0(B,S)$ mapping to the same class in $H^0(A,S)$ must come from a single class in $H^0(B,S)$ via the diagonal map. This is equivalent to saying that
\[u^*:H^0(B,S)\to H^0(A,S)\]
is injective.

The left lifting property against $\ast\sslash G\to \ast$ is equivalent to saying that 
\[u^*:Z^1(B,G)\to Z^1(A,G)\]
is surjective which implies that $H^1(B,G)\to H^1(A,G)$ is surjective. 

The left lifting against $G^c\sslash G\to (\ast\sslash G)^2$ says that if two elements of $Z^1(B,G)$ become equivalent when pull-backed to $A$, then they must already be equivalent in $Z^1(B,G)$. This is exactly saying that the map
\[H^1(B,G)\to H^1(A,G)\]
is injective.

(6) The $Q$-projective maps that are weak equivalences are $P$-projective. Let $u:A\to B$ be a map that is $Q$-projective and a weak equivalence. Being a $Q$-cofibration is equivalent to saying that for all finite group $G$, the map
\[u^*:B^1(B,G)\to B^1(A,G)\]
is surjective. 

\begin{itemize}
\item The fact that $A\to B$ is an equivalence implies that $\pi_0(A)\to\pi_0(B)$ is an isomorphism by proposition \ref{prop-equivalence induce iso on pi zero}. The lifting property against $S\to \ast$ and $S\to S\times S$ is then immediate.

\item We prove that $u$ has the left lifting property with respect to 
\[\ast\sslash G\to \ast.\]
Equivalently, we need to prove that the map
\[u^*:Z^1(B,G)\to Z^1(A,G)\]
is a surjection. Let $z$ be an element in $Z^1(A,G)$, since the map $u^*:H^1(B,G)\to H^1(A,G)$ is an isomorphism, there is $y\in Z^1(A,G)$ and $k\in B^1(B,G)$ such that $u^*(y).k=z$. Now since the map $u^*:B^1(B,G)\to B^1(A,G)$ is surjective, there is a lift $l$ of $k$ in $B^1(B,G)$. And we have
\[z=u^*(y).u^*(l)=u^*(y.l)\]
which proves that $u^*:Z^1(B,G)\to Z^1(A,G)$ is surjective.

\item Now we prove that $u$ has the left lifting property with respect to $G\sslash G\to \ast\sslash G$. Equivalently, we need to prove that the map
\[B^1(B,G)\to B^1(A,G)\times_{Z^1(A,G)} Z^1(B,G)\]
is a surjection. Let us consider a $1$-cocycle $z$ over $B$ and any coboundary $k$ over $A$ such that $u^*(z)=1.k$. This implies that $u^*(z)$ represents the base point of $H^1(A,G)$ which means that $z=1.l$ for some $l$ in $B^1(B,G)$. It is not necessarily true that $u^*(l)=k$. However, from the equation $u^*(z)=1.k$, we find that the function $u^*(l)^{-1}k$ is in $H^0(A,G)$. Since $H^0(B,G)=H^0(A,G)$, there exists $m\in H^0(B,G)$ such that $u^*(m)=u^*(l)^{-1}k$. Now we see that the element $lm$ in $B^1(B,G)$ maps to $k$ in $B^1(A,G)$ and to $1.lm$ in $Z^1(B,G)$ but $1.lm=1.l=z$ because $m$ is locally constant. 

\item Finally, we prove that $u$ has the left lifting property with respect to $G^c\sslash G\to (\ast\sslash G)^2$. According to the previous paragraph this is just saying that the map $H^1(B,G)\to H^1(A,G)$ is injective. 

\end{itemize}
\end{proof}

We also define a set $Q'$ of arrows in $f\G$. This is the set of maps:
\[\on{Codisc}(S)\to\ast\]
where $S$ is a non-empty finite set in $\mathcal{S}$.

\begin{lemm}\label{lemm-Q = Q'}
The essential image of $Q$ and $Q'$ in $f\G^{[1]}$ are the same.
\end{lemm}

\begin{proof}
In other word, we are claiming that any map in $Q$ is isomorphic to a map in $Q'$ and vice-versa. The reason this is true is that $G\sslash G$ is isomorphic to the groupoid $\on{Codisc}(G)$ where $G$ is just seen as a set. Conversely, if $S$ is non-empty, the groupoid $\on{Codisc}(S)$ is isomorphic to $G\sslash G$ for any group $G$ whose underlying set is $S$.
\end{proof}

This lemma implies that we could have used $Q'$ instead of $Q$ in the previous theorem and we would have obtained the same model structure on $\pG$.

\begin{prop}\label{prop-cofibration injective on objects}
In $\pG$ a map is a cofibration if and only if it is injective on objects.
\end{prop}

\begin{proof}
The adjunction 
\[\on{Ob}:f\G\leftrightarrows \cat{F}:\on{Codisc}\]
induces an adjunction
\[\on{Ob}:\pG\leftrightarrows \pSet:\on{Codisc}.\]
Thus, a map is injective on objects, if and only if it has the left lifting property against $\on{Codisc}(S)\to \ast$ for any non-empty finite set $S$. By lemma \ref{lemm-Q = Q'}, this is equivalent to being a cofibration.
\end{proof}

\begin{coro}
The model category $\pG$ is left proper.
\end{coro}

\begin{proof}
Any model category in which all objects are cofibrant is left proper.
\end{proof}

\subsection*{Profinite completion of groupoids}

Let $\G$ be the category of groupoids. Let $\pG$ be the category of profinite groupoids.

Let $C$ be any groupoid, then the functor $D\mapsto \cat{G}(C,D)$ from finite groupoids to sets preserves finite limits, therefore it is represented by an object $\h{C}$ in $\pG$ by remark \ref{rem-pro representable functors}. We now give a more explicit description of this completion functor. 

\begin{defi}
Let $C$ be a groupoid. An \emph{equivalence relation} on $G$ is an equivalence relation on $\on{Ob}(C)$ and an equivalence relation on $\on{Mor}(C)$ such that there exists a morphism $p:C\to E$ in $\G$ which is surjective on objects and morphisms and with the property that two objects (resp. morphisms) of $C$ are equivalent if and only they are sent to the same object (resp. morphism) of $E$.
\end{defi}

If $R$ is an equivalence relation on $C$, we denote by $C/R$ the groupoids whose objects are $\on{Ob}(C)/R$ and morphisms are $\on{Mor}(C)/R$.

\begin{prop}
Let $f:C\to D$ be a morphisms of groupoids. Then there is an equivalence relation on $C$ for which two objects (resp. morphisms) of $C$ are equivalent if and only if they are sent to the same object (resp. morphism) of $D$ by $f$. 
\end{prop}

\begin{proof}
We can just define $E$ to be the groupoid whose objects (resp. morphisms) are the objects of $D$ that are in the image of $f$. Then the map $f$ factors as
\[C\lto{p} E\to D\]
where the map $p$ is surjective on objects and morphisms. It is clear that the equivalence relation induced by $f$ is the equivalence relation induced by $p$.
\end{proof}

In the following, we call this equivalence relation the kernel of $f$ and denote it by $\on{ker}(f)$. Note that if $C=\ast\sslash G$ is a group, the data of an equivalence relation on $C$ is exactly the data of a normal subgroup of $G$. Moreover, if $f:G\to H$ is a group homomorphism, then its kernel in the group theoretic sense coincides with the kernel of the induced map $\ast\sslash G\to \ast\sslash H$.

\begin{defi}
Let $(G,R)$ be a groupoid with an equivalence relation. We say that $R$ is \emph{cofinite} if $G/R$ is in $f\G$.
\end{defi}

The set of cofinite equivalence relations on $G$ is a cofiltered poset with respect to inclusion. Therefore, we can consider the object of $\pG$ given by $\{G/R\}_{R\,\on{cofinite}}$. The following proposition shows that this is a model for $\pG$.

\begin{prop}
For any $D$ a finite groupoid, there is an isomorphism
\[\pG(\{C/R\}_{R\,\on{cofinite}},D)\cong \G(C,D)\]
which is natural in $D$.
\end{prop}

\begin{proof}
By definition of $\pG$, we have
\[\pG(\{C/R\}_{R\,\on{cofinite}},D)=\on{colim}_{R\,\on{cofinite}}f\G(C/R,D).\]

On the other hand, since any morphism $C\to D$ must have a cofinite kernel, we see that
\[\G(C,D)\cong\on{colim}_{R\,\on{cofinite}}f\G(C/R,D).\]
\end{proof}

\begin{rem}
As we have said before, an equivalence relation on a groupoid of the form $\ast\sslash G$ is exactly the data of a normal subgroup of $G$. This equivalence relation is moreover cofinite if and only if the corresponding normal subgroup is of finite index. Hence we see from the previous proposition that
\[\h{\ast\sslash G}\cong \ast\sslash \h{G}.\]
\end{rem}

\subsection*{Quillen adjunction}

The category $\G$ of groupoids has a model structure in which the cofibrations are the maps that are injective on objects, weak equivalences are the fully faithful and essentially surjective maps and the fibrations are the isofibrations. A construction can be found in section 6 of \cite{stricklandduality}

This model structure is combinatorial, proper and simplicial. We refer to this model structure as the canonical model structure.

\begin{lemm}\label{lemm-generating pro-fibrations are fibrations}
The maps in $P$ seen as maps of $\G$ are fibrations in the canonical model structure on $\G$. Similarly, the maps in $Q$ are trivial fibrations in the canonical model structure.
\end{lemm}

\begin{proof}
The trivial fibrations in the canonical model structure are the maps that are fully faithful and surjective on objects. It is thus obvious that the maps of $Q$ are trivial fibrations. The fibrations in the canonical model structure are the isofibrations, that is the map with the right lifting property against the two inclusions $[0]\to I[1]$. It is obvious that the maps $S\to\ast$ and $S\to S\times S$, $G\sslash G\to \ast\sslash G$ and $\ast\sslash G\to \ast$ have this property. The map $G^c\sslash G\to (\ast\sslash G)^2$ has this property because it can alternatively be described as the map
\[(\ast\sslash G)^{I[1]}\to(\ast\sslash G)^{[0]\sqcup [0]}\]
induced by the inclusion $[0]\sqcup[0]\to I[1]$. Since this last map is a cofibration in $\G$ and $\ast\sslash G$ is fibrant in $\G$ (as is any object) and $\G$ is a cartesian closed model category, we are done.
\end{proof}

\begin{prop}
The profinite completion functor
\[\h{(-)}:\G\to\pG\]
is a Quillen left functor.
\end{prop}

\begin{proof}
First, this functor has a left adjoint $|-|$ which sends a profinite groupoid $\{C_i\}_{i\in I}$ to $\on{lim}_IC_i$ where the limit is computed in the category $\G$. It suffices to show that the functor $|-|$ sends generating fibrations to fibrations and generating trivial fibrations to trivial fibrations but this is exactly the content of lemma \ref{lemm-generating pro-fibrations are fibrations}.
\end{proof}

In particular, since all objects are cofibrant in $\G$, the profinite completion functor preserves weak equivalences.

\begin{prop}\label{prop-profinite completion commutes with products}
Let $C$ and $D$ be two groupoids with a finite set of objects. The map 
\[\h{C\times D}\to \h{C}\times\h{D}\]
induced by the two projections $\h{C\times D}\to \h{C}$ and $\h{C\times D}\to \h{C}$ is an isomorphism.
\end{prop}

\begin{proof}
If $R$ is an equivalence relation on $C$ and $S$ is an equivalence relation on $D$, we denote by $R\times S$ the equivalence relation which is the kernel of the map
\[C\times D\to (C/ R)\times (D/ S).\]

The profinite groupoid $\h{C}\times\h{D}$ is the inverse limit of $(C\times D)/(R\times S)$ taken over all pairs $(R,S)$ of cofinite equivalence relations on $C$ and $D$. On the other hand $\h{C\times D}$ is the inverse limit of $(C\times D)/T$ taken over all cofinite equivalence relations $T$ on  $C\times D$. Thus, in order to prove the proposition, it suffices to prove that any cofinite equivalence relation $T$ of $C\times D$ is coarser than an equivalence relation of the form $R\times S$ with $R$ and $S$ cofinite equivalence relations of $C$ and $D$ respectively. 

Let $T$ be a cofinite equivalence relation on $C\times D$. We can consider the composite
\[C\to \prod_{\on{Ob}(D)}C\times[0]\to \prod_{\on{Ob}(D)}C\times D \to \prod_{\on{Ob}(D)}(C\times D)/T\]
where the first map is the diagonal map, the third map is the projection and the factor indexed by $X$ of the second map is just the product of $\id_C$ with the map $[0]\to D$ picking up the object $X$.

The kernel of this map is a cofinite equivalence relation on $C$ that we denote $T_C$. Two objects $x$ and $y$ of $C$ are $T_C$-equivalent if and only if for all $z$ in $D$, $(x,z)$ is $T$-equivalent to $(y,z)$. Likewise, two arrows $u$ and $v$ of $C$ are $T_C$-equivalent if and only if $(u,\id_z)$ is $T$-equivalent to $(v,\id_z)$ for any object $z$ of $D$. It is a cofinite equivalence relation because it is the kernel of a map with finite target (this is where we use the finiteness of the set of objects of $D$). We can define a cofinite equivalence relation $T_D$ on $D$ in a similar fashion. 

We claim that $T_C\times T_D$ is finer than $T$ indeed, if $(u,v)$ is pair of arrows that is $T_C\times T_D$-equivalent to $(u',v')$, then $(u,\id_{t(v)})$ is $T$-equivalent to $(u',\id_{t(v)})$ and similarly $(\id_{s(u)},v)$ is $T$-equivalent to $(\id_{s(u)},v')$. Thus $(u,v)=(u,\id_{t(v)})\circ (\id_{s(u)},v)$ is $T$-equivalent to $(u',v')=(u',\id_{t(v')})\circ (\id_{s(u')},v')$
\end{proof}

\begin{coro}
Let $S$ be a finite set and $G$ be a group. Then
\[\h{G[S]}\cong\h{G}[S].\]
\end{coro}

\begin{proof}
The groupoid $G[S]$ is the product $(\ast\sslash G)\times\on{Codisc}(S)$. We have already observed that $\h{\ast\sslash G}\cong \ast\sslash \h{G}$. On the other hand, since $S$ is finite, the groupoid $\on{Codisc}(S)$ is finite. The result then follows from proposition \ref{prop-profinite completion commutes with products}.
\end{proof}

Since profinite completion commutes with coproducts, this corollary gives a formula for profinite completion of groupoids with a finite set of objects in terms of profinite completion of groups.

\subsection*{More on weak equivalences}

For $C$ a groupoid, we denote by $C^{I[1]}$ the groupoid of functors from $I[1]$ to $C$. For $C=\{C_i\}_{i\in I}$ an object of $\pG$, we denote by $C^{I[1]}$ the object $\{C_i^{I[1]}\}_{i\in I}$. Note that $C^{I[1]}$ is equipped with two maps $ev_0$ and $ev_1$ to $C$ given by the evaluation at the two objects of $I[1]$. 

We say that two maps $f,g:C\to D$ in $\G$ or $\pG$ are \emph{homotopic} if there exists a map $H:C\to D^{I[1]}$ such that $ev_0\circ H=f$ and $ev_1\circ H=g$. We denote by $\pi\G$ (resp. $\pi\pG$) the category whose objects are the objects of $\G$ (resp. $\pG$) and whose morphisms are the homotopy classes of morphisms.

\begin{prop}
Let $S$ be a finite set and $A$ be an object of $\pG$, then
\[H^0(A,S)=\pi\pG(A,S).\]
Similarly, let $G$ be a finite group, then
\[H^1(A,G)=\pi\pG(A,\ast\sslash G).\]
\end{prop}

\begin{proof}
Clearly we have $\pi\pG(\{A_i\}_{i\in I},S)=\on{colim}_i\pi\pG(A_i,S)$ and similarly $\pi\pG(\{A_i\}_{i\in I},\ast\sslash G)=\on{colim}_i\pi\pG(A_i,\ast\sslash G)$, thus, it suffices to check these formulas for $A$ an object of $f\G$.

We do the case $H^1(A,G)$. The other one is similar and easier. A trivial computation shows that $(\ast\sslash G)^{I[1]}$ is isomorphic to $G^c\sslash G$ and that $(ev_0,ev_1):G^c\sslash G\to \ast\sslash G$ is exactly the map used in the definition of $H^1(A,G)$. Thus the coequalizer defining $H^1(A,G)$ is the coequalizer defining $\pi\pG(A,\ast\sslash G)$.
\end{proof}

\begin{prop}\label{prop-characterization of weak equivalences in pG pi}
A map $u:A\to B$ is a weak equivalence in $\pG$ if and only if for any finite groupoid $C$, the induced map
\[\pi\pG(B,C)\to \pi\pG(A,C)\]
is an isomorphism.
\end{prop}

\begin{proof}
Since $S$ and $\ast\sslash G$ with $S$ a finite set and $G$ a finite group are finite groupoids, we see that this is a sufficient condition for $u$ to be a weak equivalence. 

Let us prove the reverse implication. For $C$ a finite groupoid, we say that $u$ is a $C$-equivalence if the map
\[\pi\pG(B,C)\to \pi\pG(A,C)\]
is an isomorphism. Thus our goal is to prove that if $u$ is a weak equivalence, it is a $C$-equivalence for all finite groupoid $C$. By definition, $u$ is a $C$-equivalence for $C=S$ a finite set and for $C=\ast\sslash G$ with $G$ a finite group. Moreover, the class of groupoids for which $u$ is a $C$-equivalence is stable under finite products and retracts. The class of such groupoids is also stable under weak equivalences in $\G$. Indeed, in $\G$ with its canonical model structure, all objects are cofibrant and fibrant. This implies that the weak equivalences are exactly the maps that have a homotopy inverse, that is they are exactly the map that become isomorphism in $\pi\G$. In particular, the weak equivalences of $f\G$ are sent to isomorphisms in $\pi\pG$.

Let us consider a coproduct $D\sqcup E$ of finite groupoids. Let $Z=D\times E\times \{0,1\}$. We pick an object $d_0$ in $D$ and $e_0$ in $E$. There is a map $D\sqcup E\to Z$ sending $d$ to $(d,e_0,0)$ and $e$ to $(d_0,e,1)$. There is a map $Z\to D\sqcup E$ sending $(d,e,0)$ to $d$ and $(d,e,1)$ to $e$. These two maps make $D\sqcup E$ into a retract of $Z$. Thus, the class of groupoids $C$ for which $u$ is a weak equivalence is stable under finite coproducts. 

This concludes the proof since any finite groupoid is weakly equivalent to a groupoid of the form $\sqcup_{x\in X} (\ast\sslash G_x)$ where $X$ is a finite set and the $G_x$ are finite groups.
\end{proof}

\begin{coro}
Let $u:A\to B$ be a map between finite groupoid. Then $u$ is a weak equivalence in $\G$ if and only if it is a weak equivalence in $\pG$.
\end{coro}

\begin{proof}
This follows from the previous proposition, Yoneda's lemma in $\pi f\G$ and the fact that the weak equivalences in $f\G$ are exactly the homotopy equivalences (i.e. the maps that are sent to isomorphisms in $\pi f\G$). 
\end{proof}

\subsection*{Simplicial enrichment}

There is a pairing
\[\pG\times \G\to\pG\]
sending $(A,C)$ to $A\times\h{C}$. We will now prove that this is a Quillen bifunctor. The first step is to prove that this is a left adjoint in both variables. 

The compact objects of $\G$ are the groupoids $C$ with $\on{Ob}(C)$ finite and $C(x,x)$ a finitely presented group for each $x$, in particular, the objects of $f\G$ are compact. We denote by $\G_f$ the full subcategory spanned by compact objects. Note that we have an equivalence of categories $\on{Ind}(\G_f)\simeq \G$.

If $C$ is a compact object in $\G$ and $D$ is an object of $f\G$, then $D^C$ is in $f\G$. Indeed, $C$ can be written as a finite disjoint union of groupoids of the form $G[S]$ with $S$ finite and $G$ finitely presented. Thus we are reduced to proving that there are only finitely many maps from a finitely presented group to a finite group which is straightforward.

The functor $(C,D)\mapsto D^C$ from $\G_f\op\times f\G\to f\G$ preserves finite limits in both variables. Hence, it extends uniquely into a functor
\[\G\op\times\pG\to\pG\]
which preserves limits in both variable. We still denote this functor by $(C,D)\mapsto D^C$.

There is another functor $\map(-,-):\pG\op\times\pG\to \G$ given by the formula
\[\map(\{C_i\}_{i\in I},\{D_j\}_{j\in J})=\on{lim}_I\on{colim}_JD_j^{C_i}\]
where the limits and colimits are computed in the category of groupoids.

\begin{prop}
We have an isomorphism of functors of $C\in\pG$, $D\in\G$ and $E\in\pG$:
\[\pG(C\times \h{D},E)\cong \pG(C,E^D).\]
Likewise, there is a natural isomorphism
\[\pG(C\times\h{D},E)\cong \G(D,\map(C,E)).\]
\end{prop}

\begin{proof}
We prove the first isomorphism, the second is similar. Both sides preserve limits in the $E$ variable and send colimits in the $D$ variable to limits, hence, we can assume that $E$ is in $f\G$ and that $D$ is in $\G_f$. Let us write $C=\{C_i\}_{i\in I}$, then we have
\begin{align*}
\pG(C,E^D)&\cong\on{colim}_I f\G(C_i,E^D)\\
          &\cong\on{colim}_I\G(C_i\times D,E)\\
          &\cong\on{colim}_I\pG(\h{C_i\times D},E)\\
          &\cong\on{colim}_I\pG(C_i\times\h{D},E)
\end{align*}
where the last equality follows from proposition \ref{prop-profinite completion commutes with products}. Since $E$ is in $f\G$, it is cosmall in $\pG=\Pro(f\G)$, thus, we have
\[\pG(C,E^D)\cong\pG(\on{lim}_IC_i\times\h{D},E)\cong \pG(C\times\h{D},E).\]
\end{proof}

Now, we can prove the following.

\begin{prop}
The pairing
\[\pG\times \G\to \pG\]
sending $(A,C)$ to $A\times\h{C}$ is a Quillen bifunctor.
\end{prop}

\begin{proof}
We have already seen that this functor is a left adjoint in both variables. It remains to prove that is has the pushout-product property. Recall from proposition \ref{prop-cofibration injective on objects} that a map is a cofibration in $\pG$ if and only if it is injective on objects. 

(1) Let $A\to B$ be a cofibration in $\pG$ and $C\to D$ be a cofibration in $\G$, then the map
\[A\times \h{D}\sqcup^{A\times \h{C}}B\times \h{C}\to B\times \h{D}\]
is injective on object. The proof is easy once we have observed that the functor that assigns to a groupoid or profinite groupoid its set of objects is colimit preserving.

(2) If $C\to D$ is a weak equivalence in $\G$ and $A$ is any object in $\pG$, then $A\times \h{C}\to A\times \h{D}$ is a weak equivalence in $\pG$. Indeed, by proposition \ref{prop-weak equivalences stable under cofiltered limits}, we can assume that $A$ is in $f\G$. By proposition \ref{prop-characterization of weak equivalences in pG pi}, we are thus reduced to proving that
\[\pi\pG(\h{D}\times A,K)\to\pi\pG(\h{C}\times A,K)\]
is an isomorphism for all $K\in f\G$. By adjunction, we have an isomorphism $\pi\pG(\h{C}\times A,K)\cong\pi\pG(C,K^A)$ and similarly for the other side. Since $K^A$ is a finite groupoid, we are done.

(3) Similarly, if $A\to B$ is a weak equivalence in $\pG$ and $C$ is a small groupoid (i.e. an object of $\G_f$), then the map $A\times\h{C}\to B\times \h{C}$ is a weak equivalence. Indeed, by adjunction, we are reduced to proving that
\[\pi\pG(B,\map(C,K))\to\pi\pG(A,\map(C,K))\]
is an isomorphism for all $K$. But this follows from the fact that $\map(C,K)$ is finite.

(4) Now, let $A\to B$ be a trivial cofibration and $C\to D$ be a cofibration, then the map
\begin{equation}\label{e-SM7}
A\times \h{D}\sqcup^{A\times \h{C}}B\times \h{C}\to B\times \h{D}
\end{equation}
is a cofibration by (1). Moreover, the map $A\times \h{C}\to A\times \h{D}$ is a cofibration by (1) and the map $A\times \h{C}\to B\times \h{C}$ is a weak equivalence by (2). Thus, since $\pG$ is left proper, the map \ref{e-SM7} is a weak equivalence if and only if $A\times \h{D}\to B\times \h{D}$ is a weak equivalence but this follows from (2).

The other case is dealt with similarly using (3) instead of (2) and observing that we can assume that $C$ and $D$ are small groupoids (for instance because the generating trivial cofibrations of $\G$ can be chosen with small source and target).
\end{proof}

By adjunction, the functors $(C,E)\mapsto \map(C,E)$ and $(D,E)\mapsto E^D$ are also Quillen bifunctors. Since the functor $B:\G\to\S$ is a right Quillen functor, this implies that $(C,E)\mapsto B\map(C,E)$ is a Quillen bifunctor. Hence it makes $\pG$ into a simplicial model category. We denote by $\Map_{\pG}(-,-)$ the functor $B\map(-,-)$. Unwinding the definition, we see that $\Map_{\pG}(C,D)$ is the simplicial set whose $k$-simplices are $\pG(C\times I[k],D)\cong \pG(C,D^{I[k]})$.

\bigskip
Using the simplicialness of $\pG$, we can extend the criterion of proposition \ref{prop-characterization of weak equivalences in pG pi}.

\begin{prop}\label{prop-characterization of weak equivalences in pG Map}
A map $u:A\to B$ is a weak equivalence in $\pG$ if and only if for any finite groupoid $C$, the induced map
\[\Map_{\pG}(B,C)\to \Map_{\pG}(A,C)\]
is a weak equivalence.
\end{prop}

\begin{proof}
If the induced map
\[\Map_{\pG}(B,C)\to \Map_{\pG}(A,C)\]
is a weak equivalence for all $C$ in $f\G$, taking $\pi_0$ and applying proposition \ref{prop-characterization of weak equivalences in pG pi}, we find that $u$ is a weak equivalence. Conversely, notice that the finite groupoids are fibrant in $\pG$. Indeed, $\ast\sslash G$ for $G$ a finite group and $\on{Codisc}(S)$ for $S$ a finite set are fibrant by definition. Moreover, using the same trick as in the proof of proposition \ref{prop-characterization of weak equivalences in pG pi}, we find that if $U$ and $V$ are fibrant, then $U\sqcup V$ is a retract of $U\times V\times\{0,1\}$ and hence is fibrant. Since $\pG$ is a simplicial model category and $u$ is a weak equivalence between cofibrant objects, the map
\[\Map_{\pG}(B,C)\to \Map_{\pG}(A,C)\]
is a weak equivalence for any fibrant object $C$ of $\pG$ and in particular for any object of $f\G$.
\end{proof}

\subsection*{Barnea-Schlank model structure}

Given category $\cat{C}$ with a subcategory $w\cat{C}$ of weak equivalences containing all the objects, one can define $Lw^{\cong}(w\cat{C})$ to be the smallest class of arrows in $\Pro(\cat{C})$ that is stable under isomorphisms in the arrow category of $\Pro(\cat{C})$ and contains the natural transformations that are levelwise weak equivalences. 

\begin{defi}
A \emph{weak fibration category} is a triple $(\cat{C},w\cat{C},f\cat{C})$ of a small category with two subcategories containing all the objects such that :
\begin{itemize}
\item $\cat{C}$ has all finite limits.
\item The maps in $w\cat{C}$ have the two-out-of-three property.
\item The maps in $f\cat{C}$ and $f\cat{C}\cap w\cat{C}$ are stable under base change.
\item Any map has a factorization of the form $f\cat{C}\circ w\cat{C}$.
\end{itemize}
\end{defi}

\begin{theo}[Barnea-Schlank \cite{barneaprojective}]
Let $\cat{C}$ be a weak fibration category, then if $Lw^{\cong}(w\cat{C})$ satisfies the two-out-of-three property, there is a model structure on $\Pro(\cat{C})$ in which the weak equivalences are the maps of $Lw^{\cong}(w\cat{C})$, the cofibrations are the maps with the left lifting property against $f\cat{C}\cap w\cat{C}$ and the trivial cofibrations are the maps with the left lifting property against $f\cat{C}$.
\end{theo}

Note that the condition that $Lw^{\cong}(w\cat{C})$ has the two-out-of-three property is usually hard to check. Following Barnea Schlank, we call a weak fibration category satisfying this property a pro-admissible weak fibration category. There is a criterion on $\cat{C}$ that insures that $\cat{C}$ is pro-admissible.

\begin{defi}
Let $(\cat{C},w\cat{C})$ be a relative category. We say that a map $u:X\to Y$ is \emph{left proper} if any cobase change of a weak equivalence $X\to T$ along $u$ exists and is a weak equivalence. We say that $u$ is \emph{right proper} if any base change of a weak equivalence $Z\to Y$ along $u$ exists and is a weak equivalence.
\end{defi}

\begin{theo}[Barnea-Schlank]\label{theo-criterion pro-admissible}
Let $(\cat{C},w\cat{C})$ be a relative category. Let us denote by $LP$ resp. $RP$ the class of left proper (resp. right proper) maps in $\cat{C}$. Then $Lw^{\simeq}(w\cat{C})$ satisfies the two-out-of-three property if any map in $\cat{C}$ has a factorization of the form $w\cat{C}\circ LP$, $RP\circ w\cat{C}$ and $RP\circ LP$.
\end{theo}

\begin{proof}
See \cite[Proposition 3.6.]{barneaaccessibility}. Note that the authors require  that $\cat{C}$ has all finite colimits and limits. However, an inspection of the proof shows that only the pushouts of a weak equivalence along a left proper maps and the pullbacks of a weak equivalence along a right proper map are needed.
\end{proof}

The category $f\G$ of finite groupoids is a small category with finite limits. We can declare a map to be a fibration (resp. weak equivalence) if it is one in the canonical model structure on $\G$.

\begin{prop}
The category $f\G$ with this notion of fibration and weak equivalence is a pro-admissible weak fibration category.
\end{prop}

\begin{proof}
The fact that $f\G$ is a weak fibration category follows easily from the existence of the canonical model structure on $\G$. The only non-trivial axiom is the factorization axiom. Let $f:C\to D$ a map in $f\G$. Using the path object $C\mapsto C^{I[1]}$, we can factor $f$ as 
\[C\to C\times_D D^{I[1]} \to D.\]

The first map is a weak equivalence and the second map is a fibration because all objects are fibrant in $\G$. Moreover, the groupoid $C\times_D D^{I[1]}$ belongs to $f\G$. Thus, we have constructed a factorization of $f$ as a weak equivalence followed by a fibration. Note that the first map is injective on object so it is in fact a cofibration in $\G$.

Now, we want to prove that the maps in $Lw^{\simeq}(w(f\G))$ have the two-out-of-three property. We will use the criterion of theorem \ref{theo-criterion pro-admissible}. Note that $\G$ is a right proper model category. It follows that any map in $f\G$ which is a fibration in $\G$ is right proper. Similarly, if $f:C\to D$ is injective on objects and $u:C\to C'$ is a weak equivalence, the pushout $C'\sqcup^CD$ is in $f\G$, this implies that the maps that are injective on objects are left proper. Hence the previous paragraph gives us a factorization of $f$ of the form $RP\circ LP$ and $RP\circ w(f\G)$. It only remains to construct a factorization of $f$ the form $w(f\G)\circ LP$. This can be done by the following mapping cylinder construction
\[C\to (C\times I[1])\sqcup^C D \to D.\]
The first map is left proper since it is injective on objects and the second map is a weak equivalence.
\end{proof}

It follows that there is a model structure $\pG^{BS}$ on $\pG$ in which the weak equivalences are the maps in $Lw^{\simeq}(w(f\G))$, the cofibrations are the maps with the left lifting properties against the maps in $f\G$ that are fully faithful and surjective on objects and the trivial cofibrations are the maps with the left lifting property against the maps in $f\G$ that are isofibrations.

The category $f\G$ is cotensored over compact groupoids. That is, there is a pairing $(\G_f)^{op}\times f\G\to f\G$ obtained by restricting the inner Hom in groupoids: $(C,D)\mapsto D^C$. This cotensor satisfies the axiom SM7 because it does so in $\G$. More precisely, for any cofibration $u:A\to B$ in $\G_f$ and any fibration $p:E\to F$ in $f\G$, the map
\[E^B\to E^A\times_{E^A} F^B\]
is a fibration which is a weak equivalence if one of $u$ and $p$ is a weak equivalence.  Since the generating (trivial) cofibrations in $\G$ can be chosen to be in $\G_f$ and the generating (trivial) fibrations in $\pG^{BS}$ can be chosen to be in $f\G$, this implies that $\pG^{BS}$ is a $\G$-enriched model category. Applying $B$ we see that the mapping space $\Map_{\pG}$ gives $\pG^{BS}$ the structure of a simplicial model category. This argument is developed in more details in \cite{barneatwo}.

\subsection*{Equality of the two model structures}

We now want to prove that $\pG^{BS}$ and $\pG$ are in fact the same model category. Our first task is to prove that they have the same cofibrations.

\begin{lemm}
Let $u:S\to T$ be a surjective maps of finite sets. Then, there exists a set $E$ and a retract diagram
\[\xymatrix{
S\ar[d]_u\ar[r]^i& E\times T\ar[d]_{\pi_2}\ar[r]^f& S\ar[d]^u\\
T\ar[r]_{\id}&T\ar[r]_{\id}&T
}
\]
\end{lemm} 

\begin{proof}
Let $S_t$ denote the fiber of $S$ over $t$. We take $E$ to be a finite set of bigger cardinality than any of the sets $S_t$, we pick for each $t$ an injection $i_t:S_t\to E$ and a map $f_t:E\to S_t$ so that $f_t\circ i_t=\id_{S_t}$. Then we define $i:=\sqcup_t i_t:S\cong\sqcup_tS_t\to E\times T\cong\sqcup_t E$ and $f_t$ in a similar way.
\end{proof}

\begin{prop}
The closure of $Q$ under retracts, base change and composition contains the maps of $f\G$ that are trivial fibrations in $\G$.
\end{prop}

\begin{proof}
Let us denote by $E$ the closure of $Q$ under retracts, base change and composition. Let $f:C\to D$ be a trivial fibration in $\G$ between objects of $f\G$. Then $f$ is fully faithful and surjective on objects. Let us assume that $D$ is connected. In that case, $f$ can be identified non-canonically with the map
\[G[S]\to G[T]\]
induced by a surjective map $u:S\to T$. According to the previous lemma, $u$ fits in a retract diagram
\[\xymatrix{
S\ar[d]_u\ar[r]^i& E\times T\ar[d]_{\pi_2}\ar[r]^f& S\ar[d]^u\\
T\ar[r]_{\id}&T\ar[r]_{\id}&T
}
\]
It follows that the map $\on{Codisc}(S)\to\on{Codisc}(T)$ induced by $u$ is a retract of the map $\on{Codisc}(E\times T)\to\on{Codisc}(T)$ induced by $\pi_2$. The functor $\on{Codisc}$ preserves products, thus, this last map is the product of the map $\on{Codisc}(S)\to \ast$ with $\on{Codisc}(T)$. In particular, it is a base change of a map of $Q'$ hence is in $E$ by lemma \ref{lemm-Q = Q'}. This implies that $\on{Codisc}(u)$ is in $E$ and that $f$ is in $E$.

Now, let $f:C\to D$ and $f':C'\to D'$ be two maps in $E$ with non-empty target, we will show that $f\sqcup f'$ is also in $E$. Since any finite groupoid splits as a finite disjoint union of connected groupoid, this will conclude the proof. We pick an object $c$ of $C$ and $c'$ of $C'$, we call $d$ and $d'$ their image in $D$ and $D'$. Then we have a retract diagram
\[
\xymatrix{
C\sqcup C'\ar[d]\ar[r]^-{i}& C\times C'\times \{0,1\}\ar[d]\ar[r]^-{p}&C\sqcup C'\ar[d]\\
D\sqcup D'\ar[r]_-{i'}& D\times D'\times \{0,1\}\ar[r]_-{p'}&D\sqcup D'
}
\]
in which the vertical maps are induced by $f$ and $f'$, the map $i$ is the map that sends $x\in C$ to $(x,c',0)$ and $y$ in $C'$ to $(c,y,1)$, the map $p$ is the map that sends $(x,y,0)$ to $x$ and $(x,y,1)$ to $y$ and the maps $i'$ and $p'$ are defined analogously. According to this retract diagram, it suffices to prove that $f\times f'\times\{0,1\}$ is in $E$. The map $f\times f'\times\{0,1\}$ is base change of $f\times f'$, thus it suffices to show that $f\times f'$ is in $E$. But, we have
\[f\times f'=(f\times D')\circ(C\times f')\]
which implies that $f\times f'$ is in $E$.
\end{proof}

By lemma \ref{lemm-generating pro-fibrations are fibrations} the maps in $Q$ are trivial fibrations in $\pG^{BS}$. Therefore, the previous proposition tells us that $\pG^{BS}$ and $\pG$ have the same cofibrations. In order to prove that $\pG$ and $\pG^{BS}$ are the same model category, it suffices to prove that they have the same weak equivalences.

\begin{prop}
The weak equivalences of $\pG^{BS}$ are exactly the weak equivalences of $\pG$.
\end{prop}

\begin{proof}
We already know from proposition \ref{prop-weak equivalences stable under cofiltered limits} that the weak equivalences of $\pG^{BS}$ are weak equivalences in $\pG$. Conversely, let $u:A\to B$ be a weak equivalence in $\pG$. 

Let $C$ be a fibrant object of $\pG^{BS}$. Since $\pG^{BS}$ is cofibrantly generated, $C$ is a retract of $C'$ with $C'=\on{lim}_{i\in \alpha\op}C'_i$ a cocell complex with respect to the canonical fibrations in $f\G$. If $\alpha$ is a finite ordinal, this immediately implies that $C'$ is a finite groupoid and thus, by proposition \ref{prop-characterization of weak equivalences in pG Map}, this implies that the map
\[\Map_{\pG}(B,C')\to\Map_{\pG}(A,C')\]
is a weak equivalence. By proposition \ref{prop-weak equivalences stable under cofiltered limits}, we see that cofiltered limits in $\pG$ are homotopy limits. Thus, if $\alpha$ is infinite, we see that:
\[\Map_{\pG}(B,C')\simeq\on{holim}_{i\in\alpha\op}(B,C'_i)\]
and similarly for $\Map_{\pG}(A,C')$. This implies that the map induced by $u$:
\[\Map_{\pG}(B,C')\to\Map_{\pG}(A,C')\]
is a weak equivalence. Since $C$ is a retract of $C'$, the map
\[\Map_{\pG}(B,C)\to\Map_{\pG}(A,C)\]
is also a weak equivalence. This implies by Yoneda lemma in $\on{Ho}(\pG^{BS})$ that $u$ is an isomorphism in $\on{Ho}(\pG^{BS})$
\end{proof}

From this proposition, we can gives a conceptual interpretation of the $\infty$-category underlying $\pG$:

\begin{theo}
The underlying $\infty$-category of $\pG$ is equivalent to the pro-category of the $\infty$-category of $1$-truncated $\pi$-finite spaces.
\end{theo}

\begin{proof}
For $\cat{C}$ a relative category, we denote by $\cat{C}_{\infty}$ its underlying $\infty$-category. Using the main result of \cite{barneapro}, we get that $\pG_{\infty}=\pG^{BS}_{\infty}\simeq\Pro(f\G_{\infty})$. Thus, it suffices to show that $f\G_{\infty}$ is a model for the $\infty$-category of $1$-truncated $\pi$-finite spaces. The classifying space functor $f\G\to \S$ is simplicial and fully faithful, moreover it preserves weak equivalences and fibrant objects. Thus, it induces a fully faithful functor $B:f\G_{\infty}\to\S_{\infty}$. It suffices to check that any $1$-truncated $\pi$-finite space is weakly equivalent to one of the form $BC$ with $C\in f\G$. Since $B$ commutes with finite coproducts, we can restrict to connected spaces. If $X$ is a $1$-truncated connected $\pi$-finite simplicial set. Then $X\simeq B\pi_1(X,x)$ for any base point $x$ of $X$. 
\end{proof}

This also has the following corollary.

\begin{prop}\label{prop-profinite groups are fibrant}
Any profinite groupoid of the form $G[S]$ with $G$ a profinite group and $S$ a finite set is fibrant.
\end{prop}

\begin{proof}
Since $\on{Codisc}(S)\to \ast$ is in $Q'$, the object $\on{Codisc}(S)$ is fibrant in $\pG$. Hence, we are reduced to proving that $\ast\sslash G$ is fibrant in $\pG$. Let $\mathcal{N}(G)$ be the poset of open normal subgroups of $G$ of finite index. We have an isomorphism
\[\ast\sslash G\cong \on{lim}_{U\in \mathcal{N}(G)}\ast\sslash (G/U).\]

For any $U\in\mathcal{N}(G)$, the map 
\[\ast\sslash (G/U)\to\on{lim}_{U\subset V\in\mathcal{N}(G)}\ast\sslash (G/V)\]
is a fibration in $f\G$. Thus $\ast\sslash G$ is fibrant in $\pG^{BS}$ as the limit of a Reedy fibrant diagram $\mathcal{N}(G)\to \pG^{BS}$.
\end{proof}

\section{Profinite spaces}

We recall a few facts about the homotopy theory on profinite spaces. This theory was originally developed by Morel, in \cite{morelensembles}, in the pro-$p$ case and then continued by Quick in the profinite case in \cite{quickprofinite}.

\subsection*{Quick's model structure}

We denote by $\Set$ the category of sets and by $\cat{F}$ the full subcategory on finite sets. We denote by $\pSet$ the category $\Pro(\cat{F})$. 

\begin{prop}
The category $\pSet$ is copresentable.
\end{prop}

\begin{proof}
Since $\Set$ has finite limits we can apply proposition \ref{prop-co presentability of pro categories}.
\end{proof}

If $S$ is a set, the functor $\cat{F}\to\Set$ sending $U$ to $\cat{Set}(S,U)$ preserves finite limits and therefore is represented by an object $\h{S}$ in $\Pro(\Set)$ according to remark \ref{rem-pro representable functors}. There is an adjunction
\[\h{(-)}:\cat{Set}\leftrightarrows \pSet:|-|\]
where the right adjoint sends a profinite set to its limit computed in $\cat{Set}$. 

There is a more explicit description of $\h{-}$. For $S$ a set, $\h{S}$ is the pro-object $\{S/R\}_{R}$ where $R$ lives in the cofiltered poset of equivalences relations on $S$ whose set of equivalences classes is finite.

\begin{defi}
The category $\pS$ of \emph{profinite spaces} is the category of simplicial objects in $\pSet$.
\end{defi}

This category is also copresentable. There is an alternative definition of this category. We denote by $\S_{cofin}$ the category of simplicial objects in finite sets that are $k$-coskeletal for some $k$. This category $\S_{cofin}$ has all finite limits, thus the pro category $\Pro(\S_{cofin})$ is a copresentable category. There is an inclusion functor $\S_{cofin}\to\pS$ which preserves finite limits. This functor induces a limit preserving functor $\Pro(\S_{cofin})\to\pS$.

\begin{prop}
The functor $\Pro(\S_{cofin})\to\pS$ is an equivalence of categories.
\end{prop}

\begin{proof}
See \cite{barneapro}.
\end{proof}

\begin{theo}[Quick]\label{theo-Model structure on profinite spaces}
There is a fibrantly generated model structure on $\pS$. The cofibrations are the monomorphisms and the weak equivalences are the maps which induce isomorphisms on $\pi_0$ as well as on non-abelian cohomology with coefficient in a finite group and on cohomology with coefficient in a finite abelian local coefficient system. Moreover, this model structure is simplicial and left proper.
\end{theo}

\begin{proof}
See \cite[Theorem 2.1.2]{quickprofinite} with a correction in \cite[Theorem 2.3.]{quickcontinuous}\footnote{Note that there is still a small mistake in the generating fibrations in \cite{quickcontinuous}. An updated version of this paper can be found on G. Quick's webpage \url{http://www.math.ntnu.no/~gereonq/}. In this version the relevant result is theorem 2.10.}.
\end{proof}

\subsection*{Profinite completion of spaces}

There is a functor 
\[\h{(-)}:\S=\on{Fun}(\Delta\op,\Set)\to\pS=\on{Fun}(\Delta\op,\Set)\]
which is a left adjoint to the functor
\[|-|:\pS\to \S\]
obtained by applying the functor $|-|:\pSet\to\Set$ levelwise. Using the description of $\pS$ as $\Pro(\S_{cofin})$ and remark \ref{rem-pro representable functors}, we can define $\h{X}$ as the functor $\S(X,-)$ which is clearly a functor $\S_{cofin}\to\Set$ which preserves finite limits.

\begin{theo}
The adjunction
\[\h{(-)}:\S\leftrightarrows \pS:|-|\]
is a Quillen adjunction
\end{theo}

\begin{proof}
See \cite[Proposition 2.28.]{quickprofinite}.
\end{proof}

\begin{rem}
Let $\S_{\infty}$ be the $\infty$-category of spaces and $\S_{\infty}^{\pi}$ be the full subcategory spanned by $\pi$-finite spaces (i.e. spaces $X$ with $\pi_0(X)$ finite and $\pi_i(X,x)$ finite for each $x$ and $\pi_i(X,x)=0$ for $i$ big enough. In \cite{barneapro}, we prove that the underlying $\infty$-category of Quick's model structure is equivalent to $\Pro(\S_{\infty}^\pi)$ the ($\infty$-categorical) pro-category of $\S_{\infty}^{\pi}$ and that the profinite completion functor $\S\to\pS$ is a model for the functor $\S_{\infty}\to\Pro(\S_{\infty}^\pi)$ sending $X$ to $\Map(X,-)$ seen as a limit preserving functor $\S_{\infty}^\pi\to\S_{\infty}$.
\end{rem}

\subsection*{Profinite classifying space functor}

There is an adjunction
\[\pi:\S\leftrightarrows\G:B.\]

The functor $B$ is the nerve functor $\Cat\to\S$ restricted to groupoids. The functor $\pi$ sends a simplicial set $X$ to a groupoid $\pi(X)$. The groupoid $\pi(X)$ is the groupoid completion of a category $\pi'(X)$. The category $\pi'(X)$ is the free category on the graph $X_1\rightrightarrows X_0$ modulo the relation $u\circ v=w$ if there is a $2$-simplex $y$ of $X$ with $d_0(y)=u$, $d_1(y)=w$ and $d_2(y)=v$.

This adjunction restricts to an adjunction
\[\pi:\S_{cofin}\leftrightarrows f\G:B.\]

Passing to the pro-categories on both sides, we get an adjunction
\begin{equation}
\pi:\Pro(\S_{cofin})\simeq \pS\leftrightarrows \pG:B
\end{equation}

The functor $B$ preserves generating fibrations and generating trivial fibrations. Therefore it is a Quillen right functor. Hence we have a diagram of right Quillen functors which commutes up to a canonical isomorphism:
\[
\xymatrix{
\pG\ar[r]^{|-|}\ar[d]_B&\G\ar[d]^B\\
\pS\ar[r]_{|-|}&\S
}
\]

\begin{prop}\label{prop-B is fully faithful}
Let $C$ be a fibrant profinite groupoid, then the counit map
\[\pi BC\to C\]
is a weak equivalence in $\pG$
\end{prop}

\begin{proof}
It suffices to prove that for any $D$ fibrant in $\pG$, the map
\[\Map^h_{\pG}(C,D)\to \Map^h_{\pG}(\pi BC,D)\]
is a weak equivalence. By theorem \ref{theo-Quillen adjunction and mapping spaces}, it suffices to prove that the map
\[\Map^h_{\pG}(C,D)\to \Map^h_{\pS}(BC,BD)\]
is a weak equivalence. Since $\pG$ and $\pS$ are simplicial and any object is cofibrant in both model categories, it suffices to show that
\[\Map_{\pG}(C,D)\to \Map_{\pS}(BC,BD)\]
is a weak equivalence. We claim that it is in fact an isomorphism. Indeed, it suffices to show that it is an isomorphism in each degree. We have $BD^{\Delta[k]}\cong B(D^{I[k]})$. Thus it suffices to prove that $B$ is fully faithful as a functor $\pG\to \pS$ which is obvious.
\end{proof}

\subsection*{Good groupoids \`a la Serre}

For $C$ a groupoid, the unit map $C\to |\h{C}|$ induces a map $BC\to B|\h{C}|\cong |B\h{C}|$. This last map is adjoint to a map $\h{BC}\to B\h{C}$.

This map fails to be a weak equivalence in $\pS$ in general, however, it is in some cases. Let us recall, the definition of a good group due to Serre.

\begin{defi}
Let $G$ be a discrete group and $\h{G}$ be its profinite completion, we say that $G$ is \emph{good} if for any finite abelian group with a $\h{G}$-action $M$, the restriction map
\[H^*(\h{G},M)\to H^*(G,M)\]
is an isomorphism. If $C$ is a groupoid with a finite set of objects, we say that $C$ is good if each of the automorphisms of each object of $C$ is a good group.
\end{defi}

\begin{prop}\label{prop-completion of good groupoids}
Let $C$ be a good groupoid. Then the map
\[\h{BC}\to B\h{C}\]
is a weak equivalence in $\pS$.
\end{prop}

\begin{proof}
We can write $C$ as a disjoint union of groupoids of the form $G[S]$ where $G$ is good and $S$ is finite. Since completion commutes with colimits both in spaces and groupoids and $B$ preserve coproducts of groupoids, we are reduced to proving that for any good group $G$ and finite set $S$, the map
\[\h{BG[S]}\to B\h{G[S]}\]
is an equivalence in $\pS$. We have an obvious projection $G[S]\to G$ which is a weak equivalence in $\G$. We have a commutative diagram
\[\xymatrix{
\h{BG[S]}\ar[r]\ar[d]& B\h{G[S]}\ar[d]\\
\h{BG}\ar[r]& B\h{G}
}
\]
The two vertical maps are weak equivalences in $\pS$. The bottom map is a weak equivalence according to \cite[Proposition 3.6.]{quicksome}. This implies that the top map is an equivalence.
\end{proof}

We will need the following fact about good groups

\begin{prop}\label{prop-good groups extension}
Let
\[1\to N\to G\to H\to 1\]
be an exact sequence of groups in which $N$ is finitely presented and $H$ is good, then 
\begin{enumerate}
\item There is a short exact sequence of topological groups
\[1\to \h{N}\to \h{G}\to \h{H}\to 1.\]

\item If moreover $N$ is good, then $G$ is good.
\end{enumerate}
\end{prop}

\begin{proof}
This proposition is an exercise in \cite[p. 13]{serrecohomologie}. The first claim is proved in \cite[Proposition 1.2.4.]{nakamuragalois}. The second claim can be found in \cite[Proposition 1.2.5.]{nakamuragalois}
\end{proof}

\begin{coro}\label{coro-pure braid groups are good}
The pure braid groups $K_n$ are good.
\end{coro}

\begin{proof}
This is also an exercise in \cite[p.14]{serrecohomologie}. This follows from an induction applying the previous proposition to the short exact sequence
\[1\to F_n\to K_{n+1}\to K_n\to 1\]
and using the fact that the free groups are good. The details are worked out in \cite[Proposition 2.1.6]{collasthesis}.
\end{proof}

\begin{coro}\label{coro-product of good groupoids}
Let $C$ and $D$ be two good groupoids and assume that any automorphism group of $C$ is finitely presented, then $C\times D$ is a good groupoid.
\end{coro}

\begin{proof}
We can reduce to the case of groups as in the proof of proposition \ref{prop-completion of good groupoids} and then it suffices to apply proposition \ref{prop-good groups extension}.
\end{proof}

\section{Operads in groupoids}

\subsection*{The operad of parenthesized braids and the operad of parenthesized unital braids}

The object functor $\on{Ob}:\G\to\Set$ preserves product. Therefore it induces a functor $\on{Ob}:\Op\G\to\Op\Set$. Let $\oper{P}$ be an operad in $\G$ and $u:\oper{Q}\to\on{Ob}(\oper{P})$ be a morphism or operads in sets. We define $u^*\oper{P}$ to be an operad in $\Op\G$. The operads of objects of $u^*\oper{P}$ is the operad $\oper{Q}$. Given two elements $p$ and $q$ of $\oper{Q}(n)$, we define 
\[u^*\oper{P}(n)(p,q):=\oper{P}(n)(up,uq)\]
It is straightforward to verify that the collection of groupoids $u^*\oper{P}(n)$ forms an operad whose operad of objects is $\oper{Q}$. Moreover there is a map $u^*\oper{P}\to\oper{P}$ inducing the map $u:\oper{Q}\to\on{Ob}\oper{P}$ when restricting to objects.

\begin{defi}
Let $S$ be a set. We denote by $\mathbb{M}_n(S)$ the set of non-associative monomials of length $n$ in $S$. These sets are defined inductively by $\mathbb{M}_0(S)=\{\varnothing\}$, $\mathbb{M}_1(S)=S$ and $\mathbb{M}_n(S)=\bigsqcup_{p+q=n}\mathbb{M}_p(S)\times\mathbb{M}_q(S)$. We define $\mathbb{M}(S)$ the set of non-associative monomials in $S$ to be the disjoint union $\bigsqcup_{n\geq 0}\mathbb{M}_n(S)$.
\end{defi}

\begin{rem}
We denote elements of $\mathbb{M}(S)$ with the usual bracket notation for products. For instance, if $a$ and $b$ are two elements of $S$, we denote by $(a,b)$ the element of $\mathbb{M}_2(S)=S\times S$ whose first component is $s$ and second component is $b$. For longer monomials, we use nested brackets. For instance, $(a,(a,b))$ is the element of $\mathbb{M}_3(S)$ living in the summand $\mathbb{M}_1(S)\times\mathbb{M}_2(S)$ whose first coordinate is $a$ and second coordinate is $(a,b)$. Note that this element of $\mathbb{M}_3(S)$ is different from $((a,a),b)$ because the latter lives in the summand $\mathbb{M}_2(S)\times\mathbb{M}_1(S)$. We allow ourselves to drop the comas from the notation. For instance, we write $((ab)a)$ instead of $((a,b),a)$.
\end{rem}

We construct three operads in sets. The first one is just the operad $\oper{A}$ controlling the structure of an associative monoid. We have $\oper{A}(n)=\Sigma_n$. We denote an element $\sigma$ of $\Sigma_n$ by the ordered list $(\sigma(1)\ldots\sigma(n))$. We denote by $\mu\in\oper{A}(2)$ the permutation $(12)$ which is just the identity permutation of $\Sigma_2$ and by $\star$ the unique element of $\oper{A}(0)$.

The second operad $\oper{UM}$ is the operad of unital magmas. It is freely generated by an operation $\mu$ of degree $2$ and an operation $\star$ of degree $0$ satisfying the relation $\mu\circ_1\star=\mu\circ_2\star$. An algebra over the operad $\oper{UM}$ is a unital magma. That is, a set $X$ equipped with a binary operation $\bot:X\times X\to X$ and a special element $\star$ such that for all $x$ in $X$, we have $x\bot \star=\star\bot x=x$.

The set $\oper{UM}(n)$ is the subset of $\mathbb{M}_n\{1,\ldots,n\}$ on non-associative monomials containing each element of $\{1,\ldots,n\}$ exactly once. For instance $\oper{UM}(0)=\oper{UM}(1)=\ast$. $\oper{UM}(2)=\{(12),(21)\}$. The elements of $\oper{UM}(3)$ are the following
\[((12)3),(1(23)),((23)1),(2(31)),((31)2),(3(12)),((13)2),(1(32)),((32)1),(3(21)),((21)3),(2(13))\]
The elements of $\oper{UM}(4)$ are the following
\[((12)(34)),(1(2(34))),(1((23)4)),((1(23))4),(((12)3)4)\]
and any word obtained by permuting $1,2,3,4$ and keeping the parenthesization the same.

The third operad $\oper{BM}$ is the operad of based magmas. This is the free operad generated by an operation of degree $2$ denoted $\mu$ and an operation of degree $0$ denoted $\star$. An algebra for this operad is a magma equipped with a distinguished element. Contrary to the previous cases, each set $\oper{BM}(n)$ is infinite. The elements of $\oper{BM}(n)$ are the non-empty non-associative monomials in the set $\{\star,1,\ldots,n\}$ containing each of $1,\ldots,n$ exactly once. For instance $((2\star)(1(3\star)))$ is an element of $\oper{BM}(3)$ and $(\star((\star\star)\star))$ is an element of $\oper{BM}(0)$.

There are maps $\oper{BM}\to\oper{UM}\to\oper{A}$. Each map is determined by the fact that it sends the generators $\star$ and $\mu$ to $\star$ and $\mu$. For instance, the element $((2\star)(1(3\star)))$ of $\oper{BM}(3)$ is sent to $(2(13))$ by the first map and to $(213)$ by the second map. In general the first map removes all the $\star$ symbols and the possibly redundant parenthesis and the second map removes all the parenthesis except the exterior one.

We denote by $B_n$ the braid group on $n$-strands and by $K_n$ the pure braid group on $n$ strands, i.e. the kernel of the group homomorphism $B_n\to\Sigma_n$ sending a braid to its underlying permutation.

\begin{cons}\label{cons-PaB}
We define an operad $\oper{C}o\oper{B}$ in groupoids. The set of objects of the groupoid $\oper{C}o\oper{B}(n)$ is the set $\oper{A}(n)$. Given two objects $x$ and $y$ in $\oper{A}(n)=\Sigma_n$, the set of morphisms between them in $\oper{C}o\oper{B}(n)$ is the set of braids in $B_n$ whose image in $\Sigma_n$ is the permutation $y\circ x^{-1}$. Composition is given by composition of braids. We make the convention that $B_0$ is the trivial group so that $\oper{C}o\oper{B}(0)=*$. This operad is constructed in details in in \cite[Chapter 5]{fressehomotopy1}, note that Fresse denotes by $\oper{C}o\oper{B}_+$ what we denote by $\oper{C}o\oper{B}$.

Let $z:\oper{UM}\to\oper{A}$ be the map constructed above, then we define $\pab$ as $z^*\oper{C}o\oper{B}$. This operad is constructed in \cite[Chapter 6]{fressehomotopy1}. Again, what we denote $\pab$ is denoted $\pab_+$ by Fresse.

Let $u:\oper{BM}\to\oper{UM}$ be the map constructed above. We define another operad in groupoids denoted $\paub$ and called the operad of unital parenthesized braids. It is defined by the equation $\paub:=u^*\pab$.
\end{cons}

The construction of $\oper{C}o\oper{B}$ is presented here for completeness but we will not make any use of it in this work.

The group of automorphisms of any object of $\pab(n)$ is the pure braid group $K_n$. Moreover, since any two objects are isomorphic in $\pab(n)$, the groupoid $\pab(n)$ is weakly equivalent to $\ast\sslash{K_n}$. Similarly, the groupoid $\paub(n)$ is equivalent to $\ast\sslash K_n$. There is a map $v:\paub\to\pab$ which induces the map $u:\oper{BM}\to\oper{UM}$ on objects. Moreover, this map is levelwise a weak equivalence of groupoids.

The important fact for us about $\pab$ and $\paub$ is that they are groupoid models for $\oper{E}_2$. 

\begin{prop}\label{prop-PaB is E2}
Let $B:\G\to\S$ be the classifying space functor. Then the operads $B\pab$ and $B\paub$ are weakly equivalent to $\oper{E}_2$.
\end{prop}

\begin{proof}
This follows from Fiedorowicz recognition principle. The detailed argument for $\pab$ can be found in \cite[Section 3.2.]{tamarkinformality}. 
\end{proof}

\begin{rem}
The operad $\pab$ is not only a convenient model of $\oper{E}_2$. One of the reasons for its importance in mathematics is that it is exactly the operad which encodes the structure of a braided monoidal category with a strict unit. An explanation of the relationship between $\pab$ and braided monoidal categories can be found in \cite[6.2.7.]{fressehomotopy1}. The operad $\paub$ on the other hand encodes the structure of a general braided monoidal category. We do not know a reference for this fact but the proof is similar to \cite[6.2.7.]{fressehomotopy1}. The operad $\oper{C}o\oper{B}$ encodes the structure of a braided monoidal category with a strict unit and a strictly associative tensor product as is proved in \cite[Theorem 1.4.4.]{wahlthesis}. Note that this operad is denoted $\oper{C}^\beta$ by Wahl.
\end{rem}

Recall that we have a map $v:\paub\to\pab$. By precomposition, we get a map
\[\phi:\Op\G(\pab,\oper{P})\to\Op\G(\paub,\oper{P})\]

The following lemma will be useful later on.

\begin{lemm}\label{lemm-paub to P}
Assume that $\oper{P}(0)=\oper{P}(1)=\ast$. Then the map
\[\phi:\Op\G(\pab,\oper{P})\to\Op\G(\paub,\oper{P})\]
is an isomorphism.
\end{lemm}

\begin{proof}
(1) Let us first make no assumption on $\oper{P}$. Let $f:\paub\to\oper{P}$ be a map of operads in the image of $\phi$. Then we see that  if $x$ and $y$ are two objects of $\paub(n)$ that are sent to the same object of $\pab(n)$ by $v$, they must be sent to the same object of $\oper{P}(n)$ by $f$. Similarly, if a map $\alpha\in\paub(n)$ is sent the an identity map by $v$, then it must be sent to an identity map by $f$. We claim that these conditions are in fact sufficient for $f$ to be in the image of $\phi$. Indeed, if $f$ satisfies these two conditions, then, for each $n$ there exists a unique $g:\pab(n)\to\oper{P}(n)$ whose precomposition with $v$ is $f$. It is then straightforward to verify that this map $g$ is a map of operads.

(2) For $n\geq 1$, the map $u:\oper{BM}(n)\to\oper{UM}(n)$ has a section $s$ which sends a monomial in $\{1,\ldots,n\}$ to itself seen as a monomial  in $\{\star,1,\ldots,n\}$. The composite $su$ sends a non-associative monomial to the non-associative monomial obtained by removing all the $\star$ symbols and the possibly redundant brackets. For instance $((1\star)((2\star)3))$ becomes $(1(23))$. 

(3) Now, we assume that $\oper{P}(0)=\oper{P}(1)=\ast$. Let $f:\paub\to\oper{P}$ be any map. Then $\on{Ob}(f):\oper{BM}\to\on{Ob}\oper{P}$ can be extended uniquely to a map $\on{Ob}(g):\oper{UM}\to\on{Ob}\oper{P}$. Indeed, $\oper{UM}$ is obtained from $\oper{BM}$ by adding the extra relation that $\mu\circ_1\star=\mu\circ_2\star=\id$. However, this relation is automatic in $\oper{P}$ since $\oper{P}(1)$ has just one object. Similarly, let $x$ be an object in $\paub(n)$, $x$ is a non-associative monomial in $\{1,\ldots,n\}$ and $\star$. The objects $x$ and $su(x)$ are sent to the same object in $\pab(n)$. Let $\alpha$ be the trivial braid from $x$ to $su(x)$. This is sent to an identity in $\pab(n)$. Let $\gamma$, $\delta$ and $\epsilon$ be the trivial braids $(1\star)\to 1$, $(\star 1)\to 1$ and $(\star\star)\to\star$. Then $\alpha$ can be obtained from $\gamma$, $\delta$ and $\epsilon$ by repeated applications of the operation $\mu$. Therefore, if $\gamma$, $\delta$ and $\epsilon$  are sent to the identity by $f$, so is $\alpha$. But since, by assumption, $\oper{P}(1)=\oper{P}(0)=\ast$, $\gamma$, $\delta$ and $\epsilon$ must be sent to to the identity. Now, let $x$ and $y$ are two objects of $\paub(n)$ that are sent to the same object of $\pab(n)$, then we can form $su(x)$ and $su(y)$ as before. We have $su(x)=su(y)$. Moreover, we have shown that the trivial braid $x\to su(x) $ and $y\to su(y)$ are sent to identities. This immediately implies that the trivial braid $x\to y$ is sent to an identity. It follows by the argument of paragraph (1) that $f$ is in the image of the map $\phi$ which implies that $\phi$ is surjective.

(4) Let $f$ and $g$ be two maps in $\Op\G(\pab,\oper{P})$ that are sent to the same map by $\phi$. We claim that $f$ and $g$ are equal. It suffices to check that they are equal in each degree. This follows from the fact that the map $\paub(n)\to\pab(n)$ has a section. By definition of $\paub(n)$, it suffices to prove that the map $\oper{BM}(n)\to\oper{UM}(n)$ has as section but this is the content of paragraph (2). 
\end{proof}

\subsection*{Cofibrant objects in $\Op\G$}

\begin{theo}
The category of operads in $\G$ has a model structure in which the weak equivalences and fibrations are the maps that are levelwise weak equivalences and fibrations.
\end{theo}

\begin{proof}
The model category $\G$ is cartesian closed, cofibrantly generated and all objects in it are cofibrant and fibrant, therefore the hypothesis of \cite[Theorem 3.2]{bergeraxiomatic} are satisfied.
\end{proof}

Let $\on{Ob}:\cat{G}\to\Set$ be the functor sending a groupoid to its set of objects. We have already seen that $\on{Ob}$ has a left adjoint that we denote $\on{Disc}$ and a right adjoint that we denote $\on{Codisc}$. 

The functor $\on{Ob}$, $\on{Disc}$ and $\on{Codisc}$ all preserve products and we use the same notation for the functors $\Op\G\to\Op$ and $\Op\to\Op\G$ that they induce. We have two adjunctions
\[\on{Disc}:\Op\leftrightarrows\Op\G:\on{Ob},\]
\[\on{Ob}:\Op\G\leftrightarrows\Op:\on{Codisc}.\]

\begin{prop}
The cofibrations in $\Op\G$ are exactly the maps $u:\oper{A}\to\oper{B}$, such that $\on{Ob}(u)$ has the left lifting property against the maps which are levelwise surjective.
\end{prop}

\begin{proof}
Let $C$ be the class of maps $u:\oper{A}\to\oper{B}$, such that $\on{Ob}(u)$ has the left lifting property against the maps which are levelwise surjective. The functor $\on{Ob}$ has a right adjoint. This implies that the class $C$ is stable under transfinite compositions, pushouts and retracts. Moreover, the generating cofibrations are in $C$, this implies that $C$ contains the cofibrations of $\Op\G$.

Conversely, let $u:\oper{A}\to\oper{B}$ be a map of $C$. Let $p:\oper{O}\to\oper{P}$ be a trivial fibration in $\Op\G$. We have the following diagram
\[
\xymatrix{
\on{Disc}\on{Ob}\oper{A}\ar[d]\ar[r]^{\;\;c}&\oper{A}\ar[d]_u\ar[r]^f&\oper{O}\ar[d]_p\\
\on{Disc}\on{Ob}\oper{B}\ar[r]_{\;\;d}&\oper{B}\ar[r]_g&\oper{P}
}
\]
where $c$ and $d$ are the counits of the adjunction $(\on{Disc},\on{Ob})$. Since $\on{Ob}(p)$ is surjective, and $u$ is in $C$, then there is a map $k:\on{Disc}\on{Ob}\oper{B}\to\oper{O}$ such that $pk=gd$. Now we construct a map $l:\oper{B}\to\oper{O}$.

Since the map $d$ is levelwise bijective on object, we can define declare $l$ of an object of $\oper{B}(n)$ to be $k$ applied to the same object of $\on{Disc}\on{Ob}\oper{B}(n)$. The map $p$ is levelwise surjective on objects and fully faithful. Thus for any map $\alpha$ in $\oper{P}(n)$ and any choice of lift of its source and target in $\oper{O}(n)$, there is a unique map in $\oper{O}(n)$ which lifts $\alpha$ and has the chosen source and target.

Hence, for $\alpha$ an arrow of $\oper{B}(n)$, we can define $l(\alpha)$ to be the unique map lifting $g(\alpha)$ whose source is $l$ of the source of $\alpha$ and target is $l$ of the target of $\alpha$. This defines $l$ as a map of sequences of groupoids. The map $\on{Ob}(l):\on{Ob}\oper{B}\to\on{Ob}\oper{O}$ is a map of operads in sets by construction. To check that $l$ is actually a map of operads in groupoids, it suffices to check that for any arrow $\alpha$ in $\oper{B}(n)$ and $\beta$ in $\oper{B}(m)$, we have the identity
\[l(\alpha\circ_i\beta)=l(\alpha)\circ_i l(\beta).\]
The two sides of this equation are arrows in $\oper{O}(m+n-1)$ which lift $g(\alpha\circ_i\beta)$. Moreover both sides of the equation have same source and target. Therefore, by our previous observation, both sides are actually equal.
\end{proof}

\begin{coro}
Let $\oper{A}$ be an operad in $\Op\G$ such that $\on{Ob}\oper{A}$ is a free operad in sets. Then $\oper{A}$ is cofibrant.
\end{coro}

\begin{coro}\label{coro-paub cofibrant}
The operad $\paub$ is cofibrant in $\Op\G$
\end{coro}

\begin{proof}
Indeed, $\on{Ob}\pab$ is the operad $\oper{BM}$ which is free on one operation of degree $0$ and one operation of degree $2$.
\end{proof}

\section{The Grothendieck-Teichm\"uller group}

\subsection*{Drinfel'd's definition}

Recall from proposition \ref{prop-profinite completion commutes with products} that for groupoids with a finite set of objects, the profinite completion functor is symmetric monoidal. Thus there is an operad $\h{\pab}$ obtained by applying the profinite completion functor in each arity to the operad in groupoid $\pab$.

Let us recall a notation. For $i\leq n-1$, we denote by $\sigma_{i}$ the Artin generator of $B_n$. The group $B_n$ can then be defined as 
\[B_n=\langle \sigma_1,\ldots,\sigma_{n-1}|\sigma_i\sigma_{i+1}\sigma_1=\sigma_{i+1}\sigma_i\sigma_{i+1},\sigma_i\sigma_j=\sigma_j\sigma_i\;\;\mathrm{ if }\;\;|i-j|\geq 2\rangle.\]

For $i<j$, we denote by $x_{ij}$ the element of $K_n$ defined by \[x_{ij}=(\sigma_{j-1}\ldots\sigma_{i+1})\sigma_i^2(\sigma_{j-1}\ldots\sigma_{i+1})^{-1}.\]
The $x_{ij}$ generate $K_n$.

\begin{defi}\label{defi-GT}
We define \emph{the profinite Grothendieck-Teichm\"uller monoid} $\puGT$ to be the subset of elements $(\lambda,f)$ of $\h{\mathbb{Z}}\times\h{F_2}$ satisfying the following equations:

\begin{enumerate}

\item $f(x,y)f(y,x)=1$.

\item $f(z,x)z^mf(y,z)y^mf(x,y)x^m=1$ for $xyz=1$ and $m=(\lambda-1)/2$.

\item In the profinite group $\h{K_4}$ we have the equation
\[f(x_{12},x_{23}x_{24})f(x_{13}x_{23},x_{34})=f(x_{23},x_{34})f(x_{12}x_{13},x_{24}x_{34})f(x_{12},x_{23}).\]
\end{enumerate}
\end{defi}

There is a monoid structure on $\puGT$ given by 
\[(\lambda_1,f_1).(\lambda_2,f_2)=(\lambda_1\lambda_2,f_1(f_2(x,y)x^{\lambda_2}f_2(x,y)^{-1},y^{\lambda_2})f_2(x,y)).\]

The easiest way to understand this monoid structure is to consider the map
\[\puGT\to\on{End}(\h{F_2})\]
which sends $(\lambda,f)$ to the unique continuous group homomorphism
\[\h{F_2}\to\h{F_2}\]
sending $x$ to $x^\lambda$ and $y$ to $f^{-1}y^\lambda f$. The monoid structure on $\puGT$ is then such that this map is a monoid map.

\begin{theo}[Drinfel'd]
The monoid $\puGT$ is the monoid of endomorphisms of $\h{\pab}$ inducing the identity on $\on{Ob}\h{\pab}$.
\end{theo}

\begin{proof}
See \cite[Section 4]{drinfeldquasi}.
\end{proof}

We denote by $\pGT$ the group of units of $\puGT$. The previous theorem has the following immediate corollary.

\begin{coro}
The group $\pGT$ is the group of automorphisms of $\h{\pab}$ inducing the identity on objects.
\end{coro}

Since profinite completion preserves surjections of groups, the abelianization map $F_2\to \mathbb{Z}^2$ induces a surjective map of topological groups $\h{F_2}\to \h{\mathbb{Z}}^2$. 

\begin{prop}\label{prop-f is a commutator}
If $(\lambda,f)\in\puGT$, then $f$ maps to $0$ in $\h{\mathbb{Z}}^2$.				
\end{prop}

\begin{proof}
The image of $f$ in $\h{\mathbb{Z}}^2$ is of the form $\mu x+\nu y$ with $\mu$ and $\nu$ two elements in $\h{\mathbb{Z}}$. We want to prove that $\mu=\nu=0$.

Recall that we have the generators $x_{ij}$ in $K_n$. The relations between the $x_{ij}$ are all in the commutator subgroup of $K_n$ as explained in \cite[Equation 4.6, 4.7, 4.8 and 4.9]{drinfeldquasi}. This implies that the abelianization of $K_n$ is the free abelian group on $x_{ij}$ $1\leq i<j\leq n$. In particular, we have a surjective map $K_4\to\mathbb{Z}^6$ where the $6$ generators of the target are the images of $x_{12}$, $x_{13}$, $x_{14}$, $x_{23}$, $x_{24}$ and $x_{34}$. We thus get a surjective map $\h{K}_4\to\h{\mathbb{Z}}^6$.

We map the third equation defining $\puGT$ to $\h{\mathbb{Z}}^6$ via the above map $\h{K}_4\to\h{\mathbb{Z}}^6$. We get the following equation $\mu x_{12}+\nu x_{34}=0$ in the group $\h{\mathbb{Z}}^6$. It immediately implies that $\mu=\nu=0$.
\end{proof}

\subsection*{The action on $\h{K_3}$}

We have the category $\pGrp$ of profinite groups. We denote by $u\pGrp$ the category whose objects are profinite groups and whose morphisms are conjugacy classes of continuous group homomorphisms. The $u$-prefix stands for ``unbased'' as these correspond to homotopy classes of morphisms between the classifying spaces seen as unbased spaces.

\begin{lemm}\label{lemm-centralizer of x}
Let $\h{F_2}$ be the free profinite group on two generators $x$ and $y$. For any $\lambda\in \h{\mathbb{Z}}-\{0\}$, the centralizer of $x^{\lambda}$ is the subgroup generated by $x$.
\end{lemm}

\begin{proof}
This is \cite[Lemma 2.1.2.]{nakamuragalois}.
\end{proof}

\begin{prop}\label{prop-injectivity of GT to End up to homotopy}
The composite
\[\puGT\to\on{End}_{\pGrp}(\h{F_2})\to\on{End}_{u\pGrp}(\h{F_2})\]
is injective.
\end{prop}

\begin{proof}
Let $(\lambda,f)$ and $(\mu,g)$ be two elements of $\puGT$ whose image in $\on{End}(\h{F_2})$ are conjugate by some element $h\in\h{F_2}$.

We have the two equations
\[x^\lambda=h^{-1}x^\mu h\]
\[f^{-1}y^\lambda f=h^{-1}g^{-1}y^\mu gh.\]

Passing the first equation to the abelianization of $\h{F_2}$, we see that $\lambda=\mu$. This means that $h$ is in the centralizer of $x^\lambda$ which implies, according to lemma \ref{lemm-centralizer of x}, that $h=x^{\nu}$ for some $\nu$ in $\h{\mathbb{Z}}$.

The second equation informs us that $fh^{-1}g^{-1}$ is in the centralizer of $y^\mu$ which, according to lemma \ref{lemm-centralizer of x}, implies that $fh^{-1}g^{-1}=y^\rho$ for some $\rho\in\h{\mathbb{Z}}$. 

The elements $f$ and $g$ are sent to zero by the map $h{F_2}\to\h{\mathbb{Z}}^2$ by proposition \ref{prop-f is a commutator}, thus we can evaluate the equation $fh^{-1}g^{-1}=y^\rho$ in $\h{\mathbb{Z}}^2$ and we find that $\nu=\rho=0$ which in turns implies that $f=g$.
\end{proof}

Let $B_3$ be the braid group on three strands. It has two generators $\sigma_1$ and $\sigma_2$ and one relation $\sigma_1\sigma_2\sigma_1=\sigma_2\sigma_1\sigma_2$. There is a surjective map $B_3\to \Sigma_3$ sending $\sigma_1$ to the permutation $(12)$ and $\sigma_2$ to $(23)$. We can apply profinite completion to this map and we get another surjective map $\h{B}_3\to \Sigma_3$. The kernel of the map $B_3\to \Sigma_3$ is $K_3$ the pure braid group on three strands. The center of $K_3$ is an infinite cyclic group generated by $(\sigma_1\sigma_2)^3$. The elements $\sigma_1^2$ and $\sigma_2^2$ generate a free subgroup on $2$-generators and in fact there is an isomorphism $K_3\cong \mathbb{Z}\times F_2$. The kernel of the map $\h{B}_3\to\Sigma_3$ is $\h{K_3}$, the profinite completion of the pure braid group. This follows from the first claim of proposition \ref{prop-good groups extension} together with the observation that finite groups are good and that $K_3=F_2\times\mathbb{Z}$ is finitely presented. Using proposition \ref{prop-good groups extension} we also find an isomorphism $\h{K_3}\cong\h{F_2}\times\h{\mathbb{Z}}$. There is an action of $\puGT$ on $\h{B}_3$. The element $(\lambda,f)$ sends $\sigma_1$ to $\sigma_1^\lambda$ and $\sigma_2$ to $f(\sigma_2^2,\sigma_1^2)\sigma_2^\lambda f(\sigma_2^2,\sigma_1^2)$. It is easy to see that any endomorphism of $\h{B}_3$ of this form commutes with the surjection to $\Sigma_3$. Therefore, the action of $\puGT$ on $\h{B}_3$ restricts to an action on $\h{K_3}$. Using the fact that $f(x,y)=f(y,x)^{-1}$ we see that this action of $\puGT$ on $\h{K_3}$ restricts further to the standard action on $\h{F_2}$.

\begin{prop}\label{prop-GT acts faithful on K_3}
The action of $\puGT$ on $\h{K_3}$ induces an injection
\[\puGT\to\on{End}_{u\pGrp}(\h{K_3}).\]
\end{prop}

\begin{proof}
Let us denote by $\on{End}_{\pGrp}(\h{K_3}|\h{F_2})$ (resp. $\on{End}_{u\pGrp}(\h{K_3}|\h{F_2})$) the set of endomorphisms of $\h{K_3}$ which preserve $\h{F_2}\subset\h{K_3}$ (resp. the quotient of this set by the action of $\h{K_3}$ by conjugation), we have a commutative diagram where the horizontal maps are given by restriction to $\h{F_2}$.

\[\xymatrix{
\puGT\ar[r]& \on{End}_{\pGrp}(\h{K_3}|\h{F_2})\ar[r]\ar[d]&\on{End}_{\pGrp}(\h{F_2})\ar[d]\\
            &\on{End}_{u\pGrp}(\h{K_3}|\h{F_2})\ar[r]&\on{End}_{u\pGrp}(\h{F_2})
}
\]

We have proved in proposition \ref{prop-injectivity of GT to End up to homotopy} that $\puGT\to\on{End}_{u\pGrp}(\h{F_2})$ is injective. Thus, the map $\puGT\to\on{End}_{u\pGrp}(\h{K_3}|\h{F_2})$ must be injective as well. On the other hand, the injective map
\[\on{End}_{\pGrp}(\h{K_3}|\h{F_2})\to \on{End}_{\pGrp}(\h{K_3})\]
induces an injective map
\[\on{End}_{u\pGrp}(\h{K_3}|\h{F_2})\to \on{End}_{u\pGrp}(\h{K_3})\]
which concludes the proof.
\end{proof}

\subsection*{A homotopical definition of $\puGT$}

We want to give an alternative definition of $\puGT$ with a more homotopical flavor. Recall that for two profinite groupoids $C$ and $D$, a homotopy between $f$ and $g$ two maps from $C$ to $D$ is a map $h:C\to D^{I[1]}$ whose evaluation at both objects of $I[1]$ are $f$ and $g$. We denote by $\pi\pG$ the category whose objects are profinite groupoids and whose morphisms are homotopy classes of maps. 

For an operad in profinite groupoid $\oper{O}$, we can define $\oper{O}^{I[1]}$ by applying the product preserving functor $(-)^{I[1]}$ levelwise.

\begin{theo}\label{theo-main theorem naive homotopy}
The composite
\[\puGT\to\on{End}_{\Op\pG}(\h{\pab})\to\on{End}_{\pi\Op\pG}(\h{\pab})\]
is an isomorphism.
\end{theo}

\begin{proof}
Let us first prove the surjectivity\footnote{The argument for the surjectivity was explained to me by Benoit Fresse}. Since $\puGT$ is the monoid of endomorphisms of $\h{\pab}$ which induce the identity on objects, it suffices to prove that any endomorphism of $\h{\pab}$ is homotopic to one which induces the identity on objects. The operad $\on{Ob}(\h{\pab})$ is freely generated as an operad by the object $(12)$ in $\on{Ob}(\h{\pab}(2))$, thus the restriction of a morphism $u:\h{\pab}\to\h{\pab}$ on objects is entirely determined by where it sends the object $(12)$. The image of this object can be either $(12)$ in which case $u$ induces the identity on objects or $(21)$. Assume that $u(12)=(21)$. We want to construct a map $v:\h{\pab}\to\h{\pab}$ which is the identity on objects and a homotopy $h$ from $u$ to $v$. In this context, a homotopy is just a natural transformation $h$ from $u$ to $v$. Let us pick a morphism $(21)\to (12)$ in $\h{\pab}$. We define $h(12):u(12)=(21)\to v(12)=(12)$ to be this morphism. This induces in a unique way a map $h(x):u(x)\to v(x)$ for any object $x$ of $\h{\pab}(n)$ and any $n$. Now if $a:x\to y$ is a morphism in $\h{\pab}(n)$, we define $v(a)$ to be $h(y)u(a)h(x)^{-1}$. The map $v$ preserves the operad composition. Indeed, if $a:x\to y$ is a morphism in $\h{\pab}(n)$ and $b:z\to t$ is a morphism in $\h{\pab}(m)$ and $i\in\{1,\ldots,n\}$, using the fact that $u$ and $h$ preserve the operad structure, we have
\[v(a)\circ_i v(b)=h(x\circ_iz)u(a\circ_ib)h(x\circ_iz)^{-1}=v(a\circ_ib).\]

In order to prove the injectivity, it suffices to prove that the composite:
\[\puGT\to \on{End}_{\pi\Op\pG}(\h{\pab})\to \on{End}_{\pi\pG}(\h{\pab}(3))\]
is injective.

Since $\h{\pab}(3)$ is a fibrant (by proposition \ref{prop-profinite groups are fibrant}) and cofibrant object in $\pG$ which is weakly equivalent to the other cofibrant-fibrant object $\ast\sslash\h{K_3}$, it suffices to show that the composite
\[\puGT\to  \on{End}_{\pi\pG}(\h{\pab}(3))=\on{End}_{\pi\pG}(\ast\sslash \h{K_3})=\on{End}_{u\pGrp}(\h{K_3})\]
is injective but this follows from proposition \ref{prop-GT acts faithful on K_3}.
\end{proof}

\section{Proof of the main theorem}

\subsection*{Case of groupoids}

Note that if $\oper{P}$ is an operad in $\pG$, the symmetric sequence $\{\oper{P}(n)^{I[k]}\}$ has the structure of an operad in $\pG$ for any $k$. Thus, for $\oper{O}$ and $\oper{P}$ two operads in $\pG$, we can define a simplicial set $\Map_{\Op\pG}(\oper{O},\oper{P})$ whose $k$-simplices are the map of operads in profinite groupoids from $\oper{O}$ to $\oper{P}^{I[k]}$. 

\begin{prop}\label{prop-derived to underived}
There is a weak equivalence of monoids in $\S$
\[\Map_{\Op\pG}(\h{\pab},\h{\pab})\simeq \Map^h_{\WOp\pG}(\h{N^{\Psi}\pab},\h{N^{\Psi}\pab}).\]
\end{prop}

\begin{proof}
We have an isomorphism
\[\Map_{\Op\pG}(\h{\pab},\h{\pab})\cong\Map_{\Op\G}(\pab,|\h{\pab}|).\]

For any $k$, there is an isomorphisms $|\h{\pab}|^{I[k]}(1)\cong|\h{\pab}|^{I[k]}(0)=\ast$. Thus, using lemma \ref{lemm-paub to P} we find an isomorphism
\[\Map_{\Op\G}(\paub,|\h{\pab}|)\cong\Map_{\Op\G}(\paub,|\h{\pab}|).\]

Since $\paub$ is cofibrant by corollary \ref{coro-paub cofibrant} in $\Op\G$ and $|\h{\pab}|$ is fibrant, we have
\[\Map_{\Op\G}(\paub,|\h{\pab}|)\simeq \Map^h_{\Op\G}(\pab,|\h{\pab}|).\]

The functor $N^{\Psi}:\Op\G\to \WOp\G$ is a weak equivalence preserving right Quillen equivalence by proposition \ref{prop-weak operads vs operads in groupoids}, thus we have
\[\Map_{\Op\G}(\pab,|\h{\pab}|)\simeq \Map^h_{\WOp\G}(N^{\Psi}\pab,N^{\Psi}|\h{\pab}|)\simeq\Map^h_{\WOp\G}(\h{N^{\Psi}\pab},N^{\Psi}\h{\pab}).\]
Finally since $N^{\Psi}\h{\pab}$ is isomorphic to $\h{N^{\Psi}\pab}$ by proposition \ref{prop-profinite completion commutes with products}, we have a weak equivalence
\[\Map_{\Op\G}(\pab,|\h{\pab}|)\simeq\Map^h_{\WOp\G}(\h{N^{\Psi}\pab},\h{N^{\Psi}\pab}).\]
\end{proof}

This implies immediately the groupoid versions of our main theorem

\begin{theo}\label{theo-main theorem for groupoids}
The map $\puGT\to\on{End}_{\Op\pG}(\h{\pab})$
induces an isomorphism of monoids
\[\puGT\to \on{End}_{\on{Ho}\WOp\pG}(\h{N^{\Psi}\pab}).\]
\end{theo}

\begin{proof}
We can apply $\pi_0$ on both sides of the equivalence proved in the previous proposition. By theorem \ref{theo-main theorem naive homotopy}, we have an isomorphism
\[\pi_0\Map_{\Op\pG}(\h{\pab},\h{\pab})\cong\puGT.\]
\end{proof}

\subsection*{From groupoids to spaces}

We now prove our main theorem.

\begin{prop}\label{prop-from groupoids to spaces}
There is a weak equivalence of simplicial monoids
\[\Map^h_{\WOp\pS}(\h{E_2},\h{E_2})\simeq\Map^h_{\WOp\pG}(N^{\Psi}\h{\pab},N^{\Psi}\h{\pab}).\]
\end{prop}

\begin{proof}
The weak operad $N^{\Psi}B\pab$ is weakly equivalent to $E_2=N^{\Psi}\oper{E}_2$ by proposition \ref{prop-PaB is E2}, thus, we want to compute  the monoid $\Map^h_{\WOp\pS}(\h{N^{\Psi}B\pab},\h{N^{\Psi}B\pab})$.

The unit map $\pab\to |\h{\pab}|$ induces a map in $\WOp\S$:
\[N^{\Psi}B\pab\to N^{\Psi}B|\h{\pab}|\cong |N^{\Psi}B\h{\pab}| \]
where the last isomorphisms comes from the observation that $|-|$ commutes with $N^{\Psi}$ and $B$.

This map is adjoint to a map
\[\h{N^{\Psi}B\pab}\to N^{\Psi}B\h{\pab}.\]

We claim that this map is a weak equivalence in $\WOp\pS$. It suffices to check that it is a levelwise weak equivalence. For a given $T_{\bf{a}}\in\Psi$, this map is given by
\[\h{BC}\to B\h{C}\]
for some groupoid $C$ which is a finite product of groupoids of the form $\pab(n)$ for various $n$'s. Such a groupoid is good by corollary \ref{coro-product of good groupoids} and corollary \ref{coro-pure braid groups are good}. Thus, according to proposition \ref{prop-completion of good groupoids}, the map $\h{BC}\to B\h{C}$ is a weak equivalence.

Hence, we have a weak equivalence of monoids
\[\Map^h_{\WOp\pS}(\h{N^{\Psi}B\pab},\h{N^{\Psi}B\pab})\simeq \Map^h_{\WOp\pS}(N^{\Psi}B\h{\pab},N^{\Psi}B\h{\pab}).\]

There is a natural isomorphism $N^{\Psi}B\oper{O}\cong BN^{\Psi}\oper{O}$ for any operad $\oper{O}$ in profinite groupoids, therefore we have an equivalence of monoids
\[\Map^h_{\WOp\pS}(\h{N^{\Psi}B\pab},\h{N^{\Psi}B\pab})\simeq \Map^h_{\WOp\pS}(BN^{\Psi}\h{\pab},BN^{\Psi}\h{\pab}).\]

Since $BN^{\Psi}\h{\pab}$ is a weak operad, we have 
\[\Map^h_{\WOp\pS}(BN^{\Psi}\h{\pab},BN^{\Psi}\h{\pab})\simeq \Map^h_{\cat{POp}\pS}(BN^{\Psi}\h{\pab},BN^{\Psi}\h{\pab}).\]
Using the Quillen adjunction $(\pi,B)$, we have a weak equivalence
\[\Map^h_{\cat{POp}\pS}(BN^{\Psi}\h{\pab},BN^{\Psi}\h{\pab})\simeq \Map^h_{\cat{POp}\pG}(\pi BN^{\Psi}\h{\pab},N^{\Psi}\h{\pab})\]
which according to proposition \ref{prop-B is fully faithful} induces a weak equivalence
\[\Map^h_{\cat{POp}\pS}(BN^{\Psi}\h{\pab},BN^{\Psi}\h{\pab})\simeq\Map^h_{\cat{POp}\pG}(N^{\Psi}\h{\pab},N^{\Psi}\h{\pab}).\]

Finally, using the fact that $N^{\Psi}\h{\pab}$ is fibrant in $\WOp\pG$, we have proved that there is a weak equivalence
\[\Map^h_{\WOp\pS}(\h{N^{\Psi}B\pab},\h{N^{\Psi}B\pab})\simeq\Map^h_{\WOp\pG}(N^{\Psi}\h{\pab},N^{\Psi}\h{\pab}).\]
This concludes the proof.
\end{proof}

Our main theorem now follows trivially:

\begin{theo}\label{theo-main theorem spaces}
There is an isomorphism of monoids
\[\puGT\cong\on{End}_{\on{Ho}\WOp\pS}(\h{E_2}).\]
\end{theo}

\begin{proof}
According to the previous proposition, we have an isomorphism
\[\on{End}_{\on{Ho}\WOp\pS}(\h{E_2})\cong \on{End}_{\on{Ho}\WOp\pG}(N^{\Psi}\h{\pab}).\]

Using theorem \ref{theo-main theorem for groupoids}, we deduce the result.
\end{proof}

\subsection*{Higher homotopy groups}

In this subsection, we compute the higher homotopy groups of the space of homotopy automorphisms of $\h{E_2}$. First, according to proposition \ref{prop-from groupoids to spaces}, we see that the mapping space $\Map_{\WOp\pS}(\h{E_2},\h{E_2})$ is $1$-truncated (i.e. does not have homotopy groups in degree higher than $1$).

We first make the computation of the homotopy groups of the space of homotopy automorphisms of the simplicial operad $\oper{E}_2$.

\begin{theo}\label{theo-iso of E_2 topological}
The simplicial monoid $\Map^h(\oper{E}_2,\oper{E}_2)$ is weakly equivalent to the (singular complex of the) topological group $\on{O}(2,\mathbb{R})$.
\end{theo}

\begin{proof}
For $M$ a fibrant simplicial monoid, we denote by $M^{h\times}$ the inverse image of $\pi_0(M)^\times$ along the map $M\to\pi_0(M)$. This is a grouplike simplicial monoid. 

Since the category of groupoids is a simplicial model category in which all objects are cofibrant and fibrant, we have
\[\Map_{\G}(\ast\sslash\mathbb{Z},\ast\sslash\mathbb{Z})\simeq\Map_{\G}^h(\ast\sslash\mathbb{Z},\ast\sslash\mathbb{Z}).\]
Thus, using the fact that $B:\G\to\S$ is derived fully faithful, we find a weak equivalence
\[\Map_{\G}(\ast\sslash\mathbb{Z},\ast\sslash\mathbb{Z})\simeq \Map_{\S}^h(B\mathbb{Z},B\mathbb{Z}).\]
Since $B\mathbb{Z}$ is a cofibrant-fibrant model for $S^1$, this last monoid is weakly equivalent to $\Map^h_{\S}(S^1,S^1)$. It is well-known that $\Map_{\S}(S^1,S^1)^{h\times}$ is weakly equivalent to $\on{O}(2,\mathbb{R})$.

We have a weak equivalence $\pab(2)\to \ast\sslash K_2\cong \ast\sslash\mathbb{Z}$. Therefore, we also have a weak equivalence between $\on{O}(2,\mathbb{R})$ and $\Map(\pab(2),\pab(2))^{h\times}$.

We know that $B\pab\simeq \oper{E}_2$, thus we have
\[\Map^h_{\Op\S}(\oper{E}_2,\oper{E}_2)\simeq\Map^h_{\Op\S}(B\pab,B\pab).\]
We can prove exactly as in proposition \ref{prop-fully faithfulness of B} that
\[\Map^h_{\Op\S}(B\pab,B\pab)\simeq \Map^h_{\Op\G}(\pab,\pab).\]
Since $\Op\G$ is a simplicial model category and $\paub$ is cofibrant by corollary \ref{coro-paub cofibrant} and $\pab$ is fibrant, we have
\[\Map^h_{\Op\G}(\pab,\pab)\simeq\Map_{\Op\G}(\paub,\pab).\]

Since, for each $k$, $\pab^{I[k]}(0)=\pab^{I[k]}(1)=\ast$, using \ref{lemm-paub to P}, we find an isomorphism
\[\Map_{\Op\G}(\paub,\pab)\cong \Map_{\Op\G}(\pab,\pab).\]

By evaluating in degree $2$, we get a map
\[
\Map_{\Op\G}(\pab,\pab)\to\Map_{\G}(\pab(2),\pab(2)).
\]
We claim that this map induces a weak equivalence 
\begin{equation}\label{comparison}
\Map_{\Op\G}(\pab,\pab)\to\Map_{\G}(\pab(2),\pab(2))^{h\times}\simeq\on{O}(2,\mathbb{R}).
\end{equation}

Drinfel'd in \cite[Proposition 4.1.]{drinfeldquasi} proves that the monoid of endomorphism of $\pab$ which induces the identity on objects is isomorphic to $\mathbb{Z}/2$. Thus, we have a map
\[\mathbb{Z}/2\to\Map_{\Op\G}(\pab,\pab).\]
One can prove exactly as theorem \ref{theo-main theorem naive homotopy} that this map induces an isomorphism
\[\mathbb{Z}/2\cong\pi_0\Map_{\Op\G}(\pab,\pab).\]
Moreover, by definition, the non-trivial element of $\mathbb{Z}/2$ induces the unique non-trivial automorphism of $\pab(2)$ in the homotopy category of groupoids. This means that the map \ref{comparison} is an isomorphism on $\pi_0$.

Now, we want to compute the effect of \ref{comparison} on $\pi_1$. Note that according to the previous paragraph, $\Map_{\Op\G}(\pab,\pab)$ is a group-like monoid, thus it suffices to prove that the map \ref{comparison} induces an isomorphism on $\pi_1$ based at the unit. 

The group $G=\pi_1(\Map_{\Op\G}(\pab,\pab),\id)$ is the group of natural transformations of the identity map $\pab\to\pab$. More explicitly, such a natural transformation is the data of an element $h(x)\in\pab(n)(x,x)$ for each object $x$ of $\pab(n)$ and each $n$ which satisfy the relations
\begin{itemize}
\item The equation $h(x\circ_i y)=h(x)\circ_ih(y)$ holds whenever both sides are defined.
\item For all $u:x\to y$ in $\pab(n)$, we have $h(y)uh(x)^{-1}=u$
\end{itemize}

We have a map $\epsilon:G\to \mathbb{Z}=\on{Aut}_{\pab}((12))$ sending $\{h(x)\}_{x\in\on{Ob}(\h{\pab})}$ to $h((12))$. Since any object of $\h{\pab}(n)$ for any $n$ can be obtained as iterated composition of the object $(12)$, the map $G\to\mathbb{Z}$ is injective. Moreover, this map $\epsilon:G\to\mathbb{Z}$ is also the map obtained by applying $\pi_1$ to the equation \ref{comparison}. 

In order to prove that $\epsilon$ is surjective, it suffices to construct a section. We can see $K_n$ as the fundamental group of $\on{Conf}(n,\mathbb{C})$ based at $c_0=(-n+1,-n+3,\ldots,n-3,n-1)$. For each $\theta \in S^1=\mathbb{R}/2\pi\mathbb{Z}$, we can form $c_{\theta}\in\on{Conf}(n,\mathbb{C})$ to be $e^{i\theta}c_0$. This defines a map $S^1\to \on{Conf}(n,\mathbb{C})$ sending $0$ to $c_0$. Taking the fundamental group, we get a map $\mathbb{Z}\to K_n$. This maps factors through the center of $K_n$. Alternatively, the generator of $\mathbb{Z}$ gives us a natural transformation of the identity map $\ast\sslash K_n\to \ast\sslash K_n$ for each $n$. This obviously extends to a natural transformation of the identity map $\pab(n)\to\pab(n)$. All these natural transformations $\pab(n)\to\pab(n)$ are compatible with the operadic structure. A version of this statement can be found in Wahl's thesis (see \cite[Section 1.3.]{wahlthesis}) where the author proves  that a certain operad in groupoid that she denotes $\{\oper{C}_{P\beta_k}^{\beta_k}\}$ has an action of the group object in groupoids $\ast\sslash \mathbb{Z}$. The operad $\{\oper{C}_{P\beta_k}^{\beta_k}\}$ is a very close relative of the operad $\pab$, it encodes braided monoidal categories with strictly associative multiplication. It is easy to verify that the $\ast\sslash \mathbb{Z}$-action constructed by Wahl extends to a $\ast\sslash \mathbb{Z}$-action on $\pab$.

In other words, we have exhibited a map $\mathbb{Z}\to G$ and by examining what it does in degree $2$, we see that it is a section of $\epsilon$. Hence the map \ref{comparison} induces an isomorphism on $\pi_0$ and $\pi_1$. Since both sides are truncated spaces, this proves that \ref{comparison} is a weak equivalence.
\end{proof}

Now, we treat the profinite case.

\begin{theo}
The component of the identity in $\Map^h_{\WOp\S}(\h{E_2},\h{E_2})^{h\times}$ is weakly equivalent to $B|\h{\mathbb{Z}}|$.
\end{theo}

\begin{proof}
Using propositions \ref{prop-derived to underived} and \ref{prop-from groupoids to spaces}, we see that we have a weak equivalence
\[\Map^h_{\WOp\S}(\h{E_2},\h{E_2})^{h\times}\simeq\Map_{\Op\pG}(\h{\pab},\h{\pab})^{h\times}.\]
Thus, it suffices to prove that the fundamental group of $\Map^h_{\Op\pG}(\h{\pab},\h{\pab})^{k\times}$ based at the identity is $\h{\mathbb{Z}}$.

We proceed as in theorem \ref{theo-iso of E_2 topological}, there is an evaluation map
\[\Map_{\Op\pG}(\h{\pab},\h{\pab})^{h\times}\to\Map_{\pG}(\h{\pab}(2),\h{\pab}(2))^{h\times}.\]

Taking $\pi_1$, we get a map
\[\pi_1(\Map_{\Op\pG}(\h{\pab},\h{\pab}),\id)\to\pi_1( \Map_{\pG}(\h{\pab}(2),\h{\pab}(2)),\id).\]

Since $\h{\pab}(2)$ is cofibrant fibrant in $\pG$, we have
\[\Map_{\pG}(\h{\pab}(2),\h{\pab}(2))\simeq\Map^h_{\pG}(\ast\sslash\h{\mathbb{Z}},\ast\sslash\h{\mathbb{Z}}).\] 
Using the Quillen adjunction $\G\leftrightarrows\pG$, we have
\[\Map^h_{\pG}(\ast\sslash\h{\mathbb{Z}},\ast\sslash\h{\mathbb{Z}})\simeq\Map^h_{\G}(\ast\sslash\mathbb{Z},\ast\sslash|\h{\mathbb{Z}}|).\]
Since $|\h{\mathbb{Z}}|$ is commutative, the fundamental group of this last space based at the completion map $\ast\sslash\mathbb{Z}\to\ast\sslash|\h{\mathbb{Z}}|$ is isomorphic to $|\h{\mathbb{Z}}|$.

Hence, we have a map
\[\pi_1(\Map_{\Op\pG}(\h{\pab},\h{\pab}),\id)\to|\h{\mathbb{Z}}|.\]

Exactly as in theorem \ref{theo-iso of E_2 topological}, we prove that this map is injective.

In the proof of theorem \ref{theo-iso of E_2 topological}, we construct a section of this map by constructing a natural transformation of the identity map $\pab\to\pab$. In other words, we construct a map $\pab\to\pab^{I[1]}$ such that the composite of that map with the two evaluation maps $\pab^{I[1]}\to\pab$ is the identity. This natural transformation induces a natural transformation
\[\h{\pab}\to\h{\pab^{I[1}]}\]
which can be composed with the map $\h{\pab^{I[1}]}\to\h{\pab}^{I[1]}$ which is adjoint to the obvious map $\pab^{I[1]}\to |\h{\pab}|^{I[1]}\cong |\h{\pab}^{I[1]}|$. In the end, we get a map
\[\h{\pab}\to\h{\pab}^{I[1]}\]
in which both evaluation are the identity $\h{\pab}\to\h{\pab}$. Using the hom-cotensor adjunction in $\Op\G$, we have constructed a map from $\ast\sslash \mathbb{Z}$ to the groupoid of natural transformations of the identity map $\h{\pab}\to\h{\pab}$. Since the target is a profinite group, this map extends to a map
\[|\h{\mathbb{Z}}|\to\pi_1(\Map_{\Op\pG}(\h{\pab},\h{\pab}),\id).\]
It is also straightforward to check that this map is a section of the map
\[\pi_1(\Map_{\Op\pG}(\h{\pab},\h{\pab}),\id)\to|\h{\mathbb{Z}}|\]
constructed above by looking at its action in degree $2$.
\end{proof}

To summarize the previous two theorems, we have the following commutative diagram:

\[
\xymatrix{
1\ar[r]&S^1\ar[r]\ar[d]&\Map^h_{\WOp\S}(E_2,E_2)^{h\times}\ar[r]\ar[d]&\mathbb{Z}/2\ar[d]\ar[r]&1\\
1\ar[r]&\h{S^1}\ar[r]&\Map^h_{\WOp\pS}(\h{E_2},\h{E_2})^{h\times}\ar[r]&\pGT\ar[r]&1
}
\]

In this diagram, each row is a split exact sequence of grouplike simplicial monoids. The first map in each row is the inclusion of the component of the identity. The second map is the map from the group of homotopy automorphisms to its space of components. The map $\mathbb{Z}/2\to\pGT$ can be checked to be the complex conjugation map $\mathbb{Z}/2\to\on{Gal}(\bar{\mathbb{Q}}/\mathbb{Q})$ composed with the inclusion $\on{Gal}(\bar{\mathbb{Q}}/\mathbb{Q})\to\pGT$.

\subsection*{Alternative version of the main result}

Profinite completion is the left adjoint to the functor
\[|-|:\pS\to\S\]
which forgets the topology and that is the approach we chose. However, for some authors like in \cite{sullivangenetics}, profinite completion should really be the endo-functor $\S\to\S$ sending $X$ to $|R\h{X}|$ where $R$ is a fibrant replacement functor in $\pS$. It is not true that $|-|$ is fully faithful. Nevertheless, our main result remains true for this alternative definition of profinite completion.

\begin{defi}
A profinite group is \emph{strongly complete} if any normal subgroup of finite index is open. Equivalently a profinite group is strongly complete if it is isomorphic to the profinite completion of its underlying discrete group.
\end{defi}

\begin{prop}
For any $n$, the profinite group $\h{K_n}$ is strongly complete.
\end{prop}

\begin{proof}
According to the main theorem of \cite{nikolovfinitely}, it suffices to prove that $\h{K_n}$ is finitely generated as a topological group (i.e. there exists a map $F_s\to K_n$ with dense image). It is well-known that the pure braid groups $K_n$ are finitely generated. This means that for any $n$, there is an $s$ and a surjection $F_s\to K_n$. The profinite completion functor preserves surjective maps which implies that $\h{K_n}$ is finitely generated.
\end{proof}

\begin{defi}
We say that a profinite groupoid $C$ with finite set of objects is \emph{strongly complete} if for any object $x$ of $C$, the group $C(x,x)$ is a strongly complete profinite group.
\end{defi}

\begin{prop}
Let $C$ be a strongly complete profinite groupoid. Then the map $\h{|C|}\to C$ is an isomorphism.
\end{prop}

\begin{proof}
The functor $|-|$ preserves finite coproducts. The completion functor preserves finite coproducts as well since it is a left adjoint. We are therefore reduced to proving the result for a connected profinite groupoid. A connected profinite groupoid is isomorphic to $G[S]$ for some profinite group $G$. Thus, it suffices to prove the result for $G[S]$ for $S$ finite and $G$ strongly complete. But in that case, we have an isomorphism
\[\h{|G[S]|}\cong \h{|G|}[S]\cong G[S].\]
\end{proof}

\begin{prop}
Let $X$ and $Y$ be two weak operads in $\pG$ such that for all $T_{\bf{a}}$ in $\Psi$, $X(T_{\bf{a}})$ and $Y(T_{\bf{a}})$ are strongly complete, then the map
\[\Map^h_{\WOp\pS}(BX,BY)\to\Map^h_{\WOp\S}(|BX|,|BY|)\]
is a weak equivalence.
\end{prop}

\begin{proof}
This map fits in the following commutative diagram
\[
\xymatrix{
\Map^h_{\WOp\pG}(X,Y)\ar[r]\ar[d]& \Map^h_{\WOp\pS}(BX,BY)\ar[d]\\
\Map^h_{\WOp\G}(|X|,|Y|)\ar[r]&\Map^h_{\WOp\S}(|BX|,|BY|)
}
\]
The two horizontal maps are weak equivalences. Therefore, it suffices to prove that 
\[\Map^h_{\WOp\pG}(X,Y)\to\Map^h_{\WOp\G}(|X|,|Y|)\]
is a weak equivalence. Alternatively, by theorem \ref{theo-Quillen adjunction and mapping spaces}, it suffices to prove that $X\to \h{|X|}$ is a weak equivalence of weak operads in profinite groupoids. By definition of a strongly complete profinite groupoid, this is even an isomorphism.
\end{proof}

\begin{coro}\label{coro-alternative version of the main result}
Let $R\h{E_2}$ be a fibrant replacement of $\h{E_2}$ in $\WOp\pS$. Then there is a weak equivalence of monoids
\[\Map_{\WOp\pS}^h(\h{E_2},\h{E_2})\to\Map_{\WOp\S}^h(|R\h{E_2}|,|R\h{E_2}|).\]
\end{coro}

\begin{proof}
It suffices to apply the previous proposition to $X=Y=\h{N^{\Psi}\pab}$
\end{proof}

\subsection*{Remark about the $\ell$-completion}

Our result also has a pro-$\ell$ version for any prime number $\ell$. There is a model structure $\pS_\ell$ on $\pS$ due to Morel (see \cite{morelensembles}) which encodes the $\infty$-category of pro-objects in the $\infty$-category of spaces which are truncated and have homotopy groups that are finite $\ell$-groups. There is a pro-$\ell$ completion functor $X\mapsto \h{X}_\ell$ from $\S$ to $\pS_\ell$. This induces a left Quillen functor
\[\WOp\S\to \WOp\pS_\ell.\]

One can form the pro-$\ell$ completion of a group and more generally a groupoid. First, we define the category $\ell\G$ of groupoids that are finite and in which each automorphism group is an $\ell$-group. Given a groupoid $C$, its pro-$\ell$ completion is the finite limit preserving functor $D\mapsto \G(C,D)$ seen as an object of $\Pro(\ell\G)$. One can form the operad in pro-$\ell$ groupoids $\h{\pab}_{\ell}$ by applying this functor levelwise to $\pab$. We define $\puGT_{\ell}$ to be the monoid of endomorphisms of $\h{\pab}_\ell$ that induces the identity on objects. Then we have the following result whose proof is exactly the same as the proof of theorem \ref{theo-main theorem spaces}.

\begin{theo}
There is an isomorphism of monoids
\[\on{End}_{\on{Ho}\WOp\pS_\ell}(\h{E_2}_\ell)\cong \puGT_{\ell}.\]
\end{theo}

\bibliographystyle{alpha}
\bibliography{biblio}

\end{document}